\documentclass[pdflatex,sn-mathphys-num]{sn-jnl}

\usepackage{graphicx}%
\usepackage{multirow}%
\usepackage{amsmath,amssymb,amsfonts}%
\usepackage{amsthm}%
\usepackage{mathrsfs}%
\usepackage[title]{appendix}%
\usepackage{xcolor}%
\usepackage{textcomp}%
\usepackage{manyfoot}%
\usepackage{booktabs}%
\usepackage{algorithm}%
\usepackage{algorithmicx}%
\usepackage{algpseudocode}%
\usepackage{listings}%

\theoremstyle{thmstyleone}%
\newtheorem{theorem}{Theorem}
\newtheorem{proposition}[theorem]{Proposition}%

\theoremstyle{thmstyletwo}%
\newtheorem{remark}{Remark}%
\newtheorem{lemma}{Lemma}%
\newtheorem{corollary}{Corollary}%

\theoremstyle{thmstylethree}%

\usepackage{multirow}
\usepackage{amsmath}
\usepackage{color}
\usepackage{upgreek}

\usepackage{amssymb}
\usepackage{mathrsfs}
\def\Q{\mathbf{Q}}
\def\P{\mathbf{P}}
\def\I{\mathbf{I}}

\def\n{\mathbf{n}}

\def\x{\mathbf{x}}

\newcommand{\dr}{\delta \rho}

\def\_v{\mathbf{v}}

\def\d{\textrm{d}}

\newcommand{\R}{\mathbb R}

\DeclareMathOperator{\argmin}{arg\,min}

\newcommand{\nvec}{\mathbf{n}}
\newcommand{\Qvec}{\mathbf{Q}}
\DeclareMathOperator{\tr}{tr} 
\newcommand{\Ivec}{\mathbf{I}}

\begin{document}

\title[Article Title]{A Landau-de Gennes Type Theory for Cholesteric-Helical Smectic-Smectic C* Liquid Crystal Phase Transitions}


\author[1]{\fnm{Apala} \sur{Majumdar}}\email{apala.majumdar@manchester.ac.uk}

\author[2]{\fnm{Baoming} \sur{Shi}}\email{bs3705@columbia.edu}

\author[3]{\fnm{Dawei} \sur{Wu}}\email{2201110052@pku.edu.cn}

\author[4]{\fnm{Jingmin} \sur{Xia}}\email{jingmin.xia@nudt.edu.cn}

\author*[5]{\fnm{Lei} \sur{Zhang}}\email{zhangl@math.pku.edu.cn}

\affil[1]{\orgdiv{Department of Mathematics}, \orgname{University of Manchester}, \orgaddress{\city{Manchester}, \postcode{M13 9PL}, \country{UK}}}

\affil[2]{\orgdiv{Department of Applied Physics and Applied Mathematics}, \orgname{Columbia University}, \orgaddress{ \city{New York}, \postcode{10027}, \country{USA}}}

\affil[3]{\orgdiv{School of Mathematical Sciences}, \orgname{Peking University}, \orgaddress{\city{Beijing}, \postcode{100871}, \country{China}}}

\affil[4]{\orgdiv{College of Meteorology and Oceanography}, \orgname{National University of Defense Technology}, \orgaddress{\city{Changsha}, \postcode{410072}, \country{China}}}

\affil[5]{\orgdiv{Beijing International Center for Mathematical Research, Center for Quantitative Biology, Center for Machine Learning Research}, \orgname{Peking University}, \orgaddress{\city{Beijing}, \postcode{100871}, \country{China}}}


\abstract{We present a rigorous mathematical analysis of a modified Landau-de Gennes (LdG) theory modeling temperature-driven phase transitions between cholesteric, helical smectic, and smectic C* phases. This model couples a tensor-valued order parameter (nematic orientational order) with a real-valued order parameter (smectic layer modulation). We establish the existence of energy minimizers of the modified LdG energy in three dimensions, subject to Dirichlet conditions, and rigorously analyze the energy minimizers in two asymptotic limits. First, in the Oseen--Frank limit, we show that the global minimizer strongly converges to a minimizer of the Landau-de Gennes bulk energy. Second, in the limit of dominant elastic constants, we prove that the global minimizers converge to a classical helical director profile. 
Finally, through stability analysis and bifurcation theory, we derive the complete sequence of symmetry-breaking transitions with decreasing temperature—from the cholesteric phase (with in-plane twist and no layering) to an intermediate helical smectic phase (with in-plane twist and layering), and ultimately to the smectic C* phase (with out-of-plane twist and layering). These theoretical results are supported by numerical simulations.}




\maketitle

\section{Introduction}
\label{sec:intro}

Liquid crystals (LCs) are intermediate mesophases between solid and liquid states, characterized by ordered molecular arrangements \cite{de1993physics}. That is, the constituent molecules align along certain locally preferred directions, referred to as ``directors". These ordered molecular arrangements give rise to distinctive optical and electrical properties in liquid crystals, making them valuable in display technologies, optical devices, and sensors \cite{bisoyi2021liquid,edwards2020interaction,LAGERWALL20121387}. In particular, chiral LCs have shown rapid responses to a switchable applied electric field \cite{chen2024}, thus playing a vital role in the applications of display devices. Chiral LCs can exhibit different phases, such as chiral nematic (or cholesteric) and chiral smectic phases. The cholesteric phase has long-range orientation order that forms a helical twist, while the smectic phase possesses an additional layered structure with positional coherence within the layers \cite{de1972analogy,de1993physics}. There are several chiral smectic phases, such as smectic A* (or chiral smectic A) and smectic C* (or chiral smectic C) \cite{1989nature}, each with distinct characteristics. In the smectic A* phase, directors are distributing themselves averagely perpendicular to layers. In contrast, a non-zero tilt of the directors from layer normals exists in the smectic C* phase, leading to helically rotated directors along the cone surface with such tile angle \cite{2024advsci,xia2024prr,2017yadav}. By imposing an electric field, directors in the smectic A* phase can induce a tilt and transit to the smectic C* phase, resulting in a shrinking layer distance \cite{manna2008pre}.

It is mathematically challenging to model complex structural transitions in LCs. As the temperature varies, LCs exhibit different phases. We refer to the seminal work of Chen and Lubensky  \cite{chenlubensky1976} on smectic LCs; the authors proposed a mean-field model employing two coupled order parameters: the vector-valued director $\mathbf{n}$ and complex-valued function $\Psi$.
This pioneering framework has inspired numerous subsequent studies from both theoretical and experimental perspectives.
Luk'yanchuk \cite{luk1998pre} extended this model to analyze the critical temperature for the chiral nematic to smectic phase transition, while Joo and Philipps \cite{2007joo} provided a rigorous mathematical analysis of the chiral nematic to smectic C* and smectic A* transitions, with a vector-valued nematic order parameter. Garc{\'\i}a-Cervera and Joo investigate the smectic layer undulations when a magnetic field is applied \cite{garcia2012analytic}. Calderer and Joo \cite{calderer2008siap} further developed dynamical models for confined smectic C* systems, although their work focused on regimes far from the transition point.
Experimental validation came from Yoon et al.~\cite{yoon2010pre}, who systematically investigated the temperature dependence of cholesteric pitch, confirming key predictions by the Chen-Lubensky theory.
However, several limitations have become apparent in this classical approach.
From a computational perspective, the complex-valued order parameter introduces numerical challenges and potential physical inconsistencies in interpreting the imaginary part of the complex variable \cite{pevnyi2014modeling}. Moreover, as discussed by Ball and Zarnescu \cite{ball2015discontinuous}, vector-based models cannot adequately describe half-charge defects that are crucial for understanding defect-mediated phase transitions.

To overcome these limitations, we model the smectic C* phase using a modified Landau-de Gennes type tensorial model recently proposed in \cite{xia2025cstar}, which employs a tensor-valued nematic order parameter $\Qvec$ and a real-valued smectic order parameter $\delta\rho$ characterizing the density variation from the average density. This formulation offers two main advantages: (1) it is not restricted to uniaxial LC phases (as are the models with a vector-valued director) and can accurately capture the nematic-smectic coupling, including complex defect structures, and  (2) it is computationally tractable. The validity of this approach has been corroborated by both numerical and experimental studies \cite{xia2025cstar,wittmann2023prr,xia2021structural,shi2025modified}. However, at present, the model still lacks a theoretical foundation, for example, regarding how the model parameters influence the stable states, and whether the model is capable of capturing temperature-driven phase transitions in chiral smectics. 

To this end, we focus on the interpretability of this model, which is defined by multiple phenomenological coefficients - the Landau-de Gennes bulk constants, the Landau-de Gennes elastic constants for $\Qvec$, smectic bulk constants and two nematic/cholesteric-smectic coupling constants. The resulting free energy is effectively the sum of the free energy of the cholesteric phase in the Landau-de Gennes framework with additional smectic and nematic-smectic coupling contributions. We study the qualitative properties of the energy minimizers in two asymptotic limits - the Oseen-Frank limit defined by dominant Landau-de Gennes bulk constants and large macroscopic domains. In this limit, the energy minimizing $\Qvec$-profile converges strongly to a minimizer of the Landau-de Gennes bulk energy for low temperatures, in our admissible space, i.e., uniaxial with constant scalar order parameter. In other words, this can be viewed as the ``vector-valued director" limit of the modified Landau-de Gennes model, for which $\Qvec$ is almost fully parameterised by a vector-valued director. We then study the ``large elastic constant" limit of the Oseen-Frank limiting energy and recover the traditional uniaxial helical cholesteric profile, for $\Qvec$. We also provide a fairly comprehensive description of the symmetry-breaking transitions from cholesteric (chiral nematic with no smectic layering) to the helical smectic phase (chiral nematic with smectic layering) and then to the smectic C* phase (or the chiral smectic C phase) with decreasing temperature, in this phenomenological framework. This includes a rigorous stability analysis of the cholesteric phase and its loss of stability with decreasing temperature. The temperature-driven transitions are controlled by a temperature-dependent parameter in the smectic bulk energy density. These results are a step forward in the interpretability of the modified Landau-de Gennes models in \cite{xia2021structural,xia2024prr, xia2025cstar}, which is essential for studying structural phase transitions and controlling the properties of confined smectic systems. The main technical difficulties are: (i) The nematic-smectic coupling terms introduce many technical difficulties, particularly with regard to the regularity of solutions. (ii) The traditional Landau-de Gennes energy for the nematic phase admits analytic minimizers for spatially homogeneous systems. We have two coupled order parameters - $\Qvec$ and $\dr$ and the layered structure in the tensor model for smectic is described by a higher-order PDE, which makes the analysis of the symmetry-breaking transition process considerably more challenging.

The rest of this paper is organized as follows. In Section \ref{sec: tensor and vector}, we introduce the modified Landau-de Gennes model for cholesteric and smectic C* phases. Moreover, the existence of minimizers is established in three-dimensional domains, subject to Dirichlet conditions, and we study two asymptotic limits of the modified Landau-de Gennes free energy. These asymptotic limits justify the common assumption of ``uniaxial" or ``vector-valued order parameter"  in certain parameter regimes, namely, large domains with dominant nematic/cholesteric phenomenological constants. In Section \ref{sec: phase transition}, we present a detailed analysis of temperature-driven phase transitions from cholesteric to smectic C* phases, revealing the helical smectic configuration as an intermediate phase. Numerical simulations in Section \ref{sec: numerical result} provide concrete validation with a bifurcation diagram that complements our theoretical analysis. We finally draw conclusions with potential future work in Section \ref{sec:conclusions}. 

\section{Landau--de Gennes type model and its Oseen--Frank limit}\label{sec: tensor and vector}

The Landau--de Gennes (LdG) theory \cite{de1993physics}, is the most celebrated continuum theory for nematic liquid crystals (NLCs) and has been hugely successful in describing the Isotropic-Nematic (I-N) phase transition \cite{fei2018isotropic,wu2025diffuse} and structural transitions for nematics \cite{majumdar2010equilibrium,robinson2017molecular,shi2023hierarchies}. The LdG theory describes the nematic phase by the LdG $\Q$-tensor order parameter, which is a traceless and symmetric $3\times 3$ matrix denoted as
\begin{equation*}
    \Qvec=\begin{pmatrix}
     Q_{11} & Q_{12} & Q_{13}\\
     Q_{21} & Q_{22} & Q_{23}\\
     Q_{31} & Q_{32} & Q_{33}
    \end{pmatrix}.
\end{equation*}

The $\Q$ tensor is isotropic if $\Q=0$, uniaxial if $\Q$ has a pair of degenerate nonzero eigenvalues, and biaxial if $\Q$ has three distinct eigenvalues \cite{de1993physics}. A uniaxial nematic phase has a single distinguished direction of averaged molecular alignment, modeled by the eigenvector with the non-degenerate eigenvalue. A biaxial nematic phase has a primary nematic director and a secondary nematic director.  

In this paper, we use a modified LdG theory to study the smectic C*
phase (a liquid-crystal phase characterized by a layered structure in which the director is tilted with respect to the layer normal), wherein we use the LdG order parameter to describe the orientational/nematic ordering with an additional real-valued positional order
parameter $\dr$ and additional energy terms to describe the intrinsic layering of smectic phases. In the absence of surface energies, the free energy for smectic C* liquid crystals in a bounded simply connected Lipschitz domain $\Omega \in \mathbb{R}^2$ or $\Omega \in \mathbb{R}^3$ is given by \cite{xia2025cstar}
\begin{equation}\label{eq:energy_LdG}
    F(\Qvec,\delta\rho):=\int_{\Omega} \bigg\{f_{el}(\Qvec)+f_{bn}(\Qvec)+f_{bs}(\delta\rho) + f_{layer}(\Qvec,\delta\rho) + f_{angle}(\Qvec,\delta\rho) \bigg\}  ,
\end{equation}
where $\Qvec$ is the LdG order parameter (viewed as a macroscopic measure of the nematic anisotropy). The density variation $\delta\rho(\x) = \rho(\x)-\rho_0$, where $\rho_0$ is the average density, models the deviation of the molecular
density from the average molecular density $\rho_0$ at the position $\x$. In fact, the positional order parameter, $\dr$, is the real part of the classical complex-valued smectic order parameter $\psi$ proposed in \cite{de1972analogy} and we refer to \cite{pevnyi2014modeling} for further details.

The elastic energy is \cite{berreman1984tensor}
\begin{equation}
f_{el}(\Qvec) := \frac{\eta_1}{2}|\nabla\times\Qvec+2\sigma \Qvec|^2+\frac{\eta_2}{2}|\nabla\cdot\Qvec|^2+\frac{\eta_{24}}{2}(Q_{ij,k}Q_{ik,j}-Q_{ij,j}Q_{ik,k}),
\end{equation}
where $\nabla\cdot\Qvec=\partial_\alpha Q_{i\alpha}$ and $\nabla\times\Qvec=\epsilon_{i,j,k}\partial_j Q_{k\beta}$ with $\epsilon_{ijk}$ being the Levi-Civita symbol,
$\eta_1, \eta_2, \eta_{24}>0$ are elastic constants, $\sigma=2\pi/p$ with the pitch $p$ of the cholesteric helix for chiral nematics, and $\sigma=0$ for an achiral nematic.
The bulk energy with respect to $\Qvec$ is given by
\begin{equation}\label{f_B}
    f_{bn}(\Q): = \frac{A}{2}\mathrm{tr} \Q^2 - \frac{B}{3} \mathrm{tr} \Q^3 + \frac{C}{4} (\mathrm{tr} \Q^2)^2-f_{B,0},
\end{equation}
where $A=\alpha_1 (T - T_1^*)$ is the rescaled temperature with $\alpha_1>0$ and $T_1^*$ a characteristic temperature related to the loss of stability of the isotropic phase; $B, C>0$ material-dependent bulk constants. For example, typical values for the representative nematic LC material MBBA are $B=0.64\times10^4 \text{Nm}^{-2}$, $C=0.35\times10^4 \text{Nm}^{-2}$ and $K=4\times10^{-11} \text{N}$ \cite{majumdar2010equilibrium,yin2020construction,shi2022nematic}. The constant $f_{B,0}$ is added to ensure a non-negative energy density. The minimizers of the bulk energy $f_{bn}$ depend on the temperature $A$ and determine the nematic phase for spatially homogeneous samples. Specifically, the minimizers of $f_{bn}$ are in the isotropic state when $A> \frac{B^2}{27 C}$, whereas for $A<\frac{B^2}{27 C}$, they constitute a continuum of $\Q$-tensor defined below:

\[
\mathcal{N} =  \left\{ \Q = s_+\left(\mathbf{n}\otimes \mathbf{n} - \frac{\mathbf{I_3}}{3} \right) \right\},
\]
where
\begin{equation}\label{eq: s+}
s_+ = \frac{B + \sqrt{B^2 - 24 AC}}{4C}, 
\end{equation}
$\mathbf{n}$ is an arbitrary unit vector field (also referred to as the nematic director \cite{de1993physics, stewart2019static}), and $\mathbf{I}_m$ is the $m\times m$ identity matrix with the spatial dimension $m$.

The third term in the total free energy (\ref{eq:energy_LdG}) is the bulk energy density of the smectic order parameter $\dr$, which is derived from the Landau theory of phase transitions \cite{de1972analogy,pevnyi2014modeling,izzo2020landau}:
\begin{equation}\label{bs term}
f_{bs}(\dr)=\frac{d}{2}(\dr)^2 - \frac{e}{3} (\dr)^3 +  \frac{f}{4} (\dr)^4,
\end{equation}
where $d=\alpha_2(T-T^*_2)$ is a temperature-dependent parameter with $\alpha_2>0$, and $T_2^*<T_1^*$ is a critical material temperature related to the nematic-smectic phase transition; $e,f>0$ are material-dependent constants. Typically, a non-zero $e$ will result in asymmetric layer structures \cite{pevnyi2014modeling}, and we take $e=0$ to study symmetric layer structures throughout this paper. When $d<0$, i.e., the temperature is low enough, the minimizers of $f_{bs}$ prefer a non-zero density distribution, hence the smectic phase.
 
Furthermore, the energy density for the smectic layer structure is given by
\begin{equation}
f_{layer}(\Qvec,\delta\rho) = \lambda_1\left(\Delta\delta\rho + q^2\delta\rho\right)^2.
\end{equation}
This energy density is minimized when $\delta\rho(\mathbf x) = \sin(q\,\mathbf n_{normal}\cdot\mathbf x)$,
with an arbitrary unit layer normal $|\mathbf n_{normal}|=1$.
Consequently, $q$ is often identified as the wave number of the SmA layers
\cite{shi2025modified,pevnyi2014modeling}, and is expected to be related to the
layer thickness $l$ via $q = 2\pi/l$. The penultimate term in the total energy \eqref{eq:energy_LdG} controls the angle between the layer normals, $\mathbf n_{normal}$, and directors $\n$, and is given by
\begin{equation}\label{eq: fangle}
f_{angle}(\Qvec,\delta\rho) = 
\lambda_2 \left(\mathrm{tr}\left({D}^2\delta\rho \left(\mathbf{Q}+\frac{\Ivec}{3}\right)\right) + q^2\delta\rho\cos^2\theta_0\right)^2,
\end{equation}
where ${D}^2$ implies the Hessian operator, and $\theta_0\geqslant 0$ is the preferred angle between the layer normal and the director. For instance, $\theta_0$ retains a non-zero value for the smectic C phase, while it drops to a zero for the smectic A phase, as shown in \cite{xia2024prr}.

\subsection{Existence of global minimizers}
Consider a bounded simply connected Lipschitz domain $\Omega \in \mathbb{R}^3$. Guided by the structure of the modified LdG energy in \eqref{eq:energy_LdG}, the admissible $\Q$-tensor and smectic order parameter $\dr$ belong to Hilbert spaces
\begin{equation}
W^{1,2}_{\mathbf{S}_0}(\Omega)=\left\{\Q \in \mathbf{S}_0|Q_{i,j} \in W^{1,2}(\Omega)\right\},
\end{equation}
and $W^{2,2}(\Omega)$, respectively, where
\begin{equation}
\begin{aligned}
\mathbf{S}_0:=&\left\{\Q \in \mathbb{R}^{3\times 3}: \Q_{ij}=\Q_{ji},\sum_{i = 1}^3 \Q_{ii}=0\right\},\\
 W^{k,p}(\Omega)=& \left\{ \dr:  \int_\Omega \left\{|\dr|^p+\sum_{|\alpha|\leqslant k}| {D}^\alpha \dr|^p \right\} <\infty \right\}.
\end{aligned}
\end{equation}
In this subsection, we consider the Dirichlet boundary condition $\Q=\Q_{bc}, \dr=\dr_{bc}, \text{on } \partial \Omega$, where $\Q_{bc} \in W^{\frac{1}{2},2}(\Omega)$ and $\dr_{bc} \in W^{\frac{3}{2},2}(\Omega)$, as in \cite{xia2023variational}. Dirichlet boundary conditions correspond to strong anchoring on the boundaries, for example, the photo-patterned anchoring can control the surface alignment \cite{wu2023topological,wu2022electrically}, which can control the nematic and potentially, smectic ordering on the boundaries. Here, $\Q_{bc}$ is uniaxial of the form, $\Q_{bc}=s_+(\n_{bc}\otimes \n_{bc}-\I_3/3)$, on $\partial\Omega$.
Then, the admissible spaces are
\begin{equation}
\begin{aligned}
W_{\Q}&=W^{1,2}_{\mathbf{S}_0}(\Omega) \cap \left\{\Q: \Q=\Q_{bc} \text{ on } \partial \Omega \right\} ,\\
W_{\dr}&=W^{2,2}(\Omega) \cap \left\{\dr: \dr=\dr_{bc} \text{ on } \partial \Omega \right\}.
\label{eq: admissible Dirichlet for u}
\end{aligned}
\end{equation}

Following the direct method of calculus of variations, we prove the existence of global minimizers of \eqref{eq:energy_LdG} by establishing several essential lemmas.

\begin{lemma}
If $\eta_1>0,0<\eta_{24}<3\eta_1,5\eta_1+10\eta_2-9\eta_{24}>0$, then there exists a constant $C_0>0$ subject to
\begin{equation}\frac{\eta_1}{2}|\nabla \Q|^2+\frac{\eta_2-\eta_{24}}{2}Q_{ij,j}Q_{ik,k}+\frac{\eta_{24}-\eta_1}{2}Q_{ij,k}Q_{ik,j}\geqslant C_0 |\nabla \Qvec|^2.
\label{eq: positive quadratic nabla Q}
\end{equation}
\end{lemma}

\begin{proof}
From \cite{wang2021modeling}, $L_1|\nabla \Q|^2+L_2 Q_{ij,j}Q_{ik,k}+L_3 Q_{ij,k}Q_{ik,j}$ is the
positive definite quadratic form of $|\nabla \Qvec|^2$ if $L_1>0, -L_1<L_3<2L_1, L_1+\frac{5}{3}L_2+\frac{1}{6}L_3>0$. Substituting $L_1=\frac{\eta_1}{2}, L_2=\frac{\eta_2-\eta_{24}}{2}, L_3=\frac{\eta_{24}-\eta_1}{2}$ completes the proof.
\end{proof}

\begin{lemma}
Assume that the boundary $\partial \Omega \in C^{1,1}$ and $p>1$, then there exists a constant $C(\Omega,p)$ depending on $\Omega$ and $p$, such that
\begin{equation}
\Vert {D}^2 \dr \Vert_{L^p(\Omega)} \leqslant C(\Omega,p)\left(\Vert \Delta \dr \Vert_{L^p(\Omega)} + \Vert \dr \Vert_{L^p(\Omega)} +\Vert \dr_{bc} \Vert_{W^{2-\frac{1}{p},p}(\partial \Omega)}\right).
\label{eq: Wp estimation}
\end{equation}
\end{lemma}
\begin{proof}
This follows directly from the global \(W^{2,p}(\Omega)\) estimate for harmonic equations (see \cite{evans2022partial}).

\end{proof}

\begin{lemma}\label{lemma: wsls}
$F(\Qvec,\delta\rho)$ is weakly lower semi-continuous (w.l.s.c.), i.e.,
    \begin{equation}    \liminf_{k\rightarrow \infty}F(\Q_k,\dr_k)\geqslant F(\Q,\dr)\text{ holds if } (\Q_k,\dr_k)\rightharpoonup (\Q,\dr) \text{ in } W_\Q \times W_\dr
    \end{equation}
for $\eta_1, \eta_2$, and $\eta_{24}$ satisfying the following parameter relation 
\begin{equation}
\eta_1>0,0<\eta_{24}<3\eta_1,5\eta_1+10\eta_2-9\eta_{24}>0.
\label{eq: positive definite domain}
\end{equation}
\end{lemma}

\begin{proof}
We first consider the LdG elastic energy density, 
$$
f_{el}(\Q)=\frac{\eta_1}{2}|\nabla \times\Qvec+2\sigma \Qvec|^2+\frac{\eta_2}{2}|\nabla \cdot\Qvec|^2+\frac{\eta_{24}}{2}(Q_{ij,k}Q_{ik,j}-Q_{ij,j}Q_{ik,k}).
$$

Note that $|\nabla \times\Qvec|^2=|\nabla \Qvec|^2-Q_{ij,k}Q_{ik,j}$, $|\nabla \cdot \Q|^2=Q_{ij,j}Q_{ik,k}$, we have
\begin{equation}
\label{eq:refom_felastic}
\begin{aligned}
f_{el}(\Q)
&=\frac{\eta_1}{2}|\nabla \Q|^2+\frac{\eta_2-\eta_{24}}{2}Q_{ij,j}Q_{ik,k}+\frac{\eta_{24}-\eta_1}{2}Q_{ij,k}Q_{ik,j}\\
&\quad+2\eta_1\sigma\epsilon_{ikl}Q_{lj,k}Q_{ij}+2\eta_1\sigma^2|\Q|^2,
\end{aligned}
\end{equation}
where $\epsilon_{ijk}=(i-j)(j-k)(k-i)/2$ is the Levi-Civita symbol. By the Rellich–Kondrachov theorem, $\Q_k \rightarrow \Q \text{ in } L^2(\Omega)$, $\dr_k \rightarrow \dr \text{ in } L^2(\Omega)$, $
\nabla \Q_k \rightharpoonup \nabla \Q \text{ in } L^2(\Omega)$, $D^2 \dr_k \rightharpoonup D^2 \dr \text{ in } L^2(\Omega)$, and thus, $\int_\Omega \{ 2\eta_1\sigma\epsilon_{ikl}Q_{lj,k}Q_{ij}+2\eta_1\sigma^2|\Q|^2 \}$ is w.l.s.c..
As the elastic constant satisfies the relation \eqref{eq: positive definite domain}, the first three terms in \eqref{eq:refom_felastic} are the positive definite quadratic form of $\nabla \Q$ from \eqref{eq: positive quadratic nabla Q}, and thus $f_{el}$ is w.l.s.c.. The weak lower semi-continuity of the nematic bulk energy density $f_{bn}$ is guaranteed in \cite{davis1998finite, canevari2017order,shi2024multistability}. Similarly, the weak lower semi-continuity of the smectic bulk energy density $f_{bs}$ can be proven. Now, consider the nematic-smectic coupling term (i.e., $f_{layer}+f_{angle}$)
\begin{equation}
\begin{aligned}
\int_\Omega L(D^2 \dr,\dr,\Q) := \int_\Omega & \bigg\{\lambda_1 (\Delta \dr + q^2\dr)^2 \\&+\lambda_2 \left(\mathrm{tr}\left(D^2\delta\rho \left(\mathbf{Q}+\frac{\Ivec_3}{3}\right)\right) + q^2\delta\rho\cos^2\theta_0\right)^2\bigg\}.
\end{aligned}
\label{eq: layer thickness term}
\end{equation}
The coupling energy density $L(D^2 \dr,\dr,\Q)\geqslant 0$,  is obviously bounded from below. 
Moreover, the terms involving derivatives are convex with respect to $D^2 \dr$, as they are semi-positive-definite quadratic polynomials. Thus, we can deduce that the nematic-smectic coupling term \eqref{eq: layer thickness term} is also w.l.s.c.~by \cite[Theorem 1 in Section 8.22]{evans2022partial}. 
\end{proof}

\begin{proposition}\label{existence in Q model}
Given that $\Omega \in \R^3$, $\partial \Omega \in C^{1,1}$ and the elastic constants $\eta_i$ satisfy \eqref{eq: positive definite domain}, with fixed $A$, positive $B, C$, $d, e, f>0, q, \sigma, \lambda_1>0, \lambda_2>0, 0<\theta_0<\pi/2,$ the energy functional \eqref{eq:energy_LdG} has at least a global minimizer $(\Tilde{\Q},\Tilde{\dr})$ in $W_{\Q}\times W_\dr$.
\end{proposition}
\begin{proof}
It is obvious that the admissible space $W_\Q\times W_{\dr}$ is non-empty from the trace theorem \cite{adams2003sobolev}. The elastic and bulk energy of $\Q$ is bounded from below by
\begin{equation}
\begin{aligned}
 &\int_\Omega \left\{ \frac{\eta_1}{2}|\bar{\nabla}\times\Qvec+2\sigma \Qvec|^2+\frac{\eta_2}{2}|\bar{\nabla}\cdot\Qvec|^2+\frac{\eta_{24}}{2}(Q_{ij,k}Q_{ik,j}-Q_{ij,j}Q_{ik,k})+f_{bn}(\Q)\right\} 
\\
&=\int_\Omega \bigg\{\underbrace{\frac{\eta_1}{2}|\nabla \Q|^2+\frac{\eta_2-\eta_{24}}{2}Q_{ij,j}Q_{ik,k}+\frac{\eta_{24}-\eta_1}{2}Q_{ij,k}Q_{ik,j}}_{\geqslant C_0 |\nabla \Qvec|^2 \text{ from } \eqref{eq: positive quadratic nabla Q}}\\
&\quad \quad+\underbrace{2\eta_1\sigma\epsilon_{ikl}Q_{lj,k}Q_{ij}+2\eta_1\sigma^2|\Q|^2}_{\geqslant -\epsilon |\nabla \Q|^2-C_1(\epsilon)|\Qvec|^2, \text{Young's inequality}} +f_{bn}(\Q) \bigg\}  \\
 &\geqslant C_0\Vert \nabla \Q\Vert_{L^2_{\mathbf{S}_0}(\Omega)}^2-\epsilon \Vert \nabla \Q\Vert_{L^2_{\mathbf{S}_0}(\Omega)}^2 +\int_\Omega \left\{f_{bn}(\Q)-C_1(\epsilon) |\Q|^2\right\} .
\end{aligned}
\label{eq: coercive with respect to Q}
\end{equation}
Note that $f_{bn}(\Q)-C_1(\epsilon) |\Q|^2$ is in fact a fourth-order polynomial of $\Q$, which is bounded form below as we take the model parameter $C>0$ in \eqref{f_B}. Thus, there exist two positive constants $C_2(A,B,C,\epsilon), C_3(A,B,C,\epsilon)>0$, such that
\begin{equation}
 \int_\Omega \left\{f_{bn}(\Q)-C_1(\epsilon) |\Q|^2\right\}   \geqslant C_2(A,B,C,\epsilon)\Vert \Q \Vert^2_{L^2_{\mathbf{S}_0}(\Omega)}-C_3(A,B,C,\epsilon).
\end{equation}
By choosing a small enough $0<\epsilon<C_0$, we can deduce from \eqref{eq: coercive with respect to Q} that $f_{el}(\Q)+f_{bn}(\Q)$ in \eqref{eq:energy_LdG} is coercive with respect to $\Q$. 
Now we proceed to proving the coerciveness related to $\dr$. The bulk density $f_{bs}(\dr)$ is a fourth order polynomial of $\dr$ with $f>0$, and $\int_\Omega f_{bs}(\dr)$ is bounded from below, then $\Vert \dr \Vert_{{L^2}(\Omega)}$, $\Vert (\dr)^2 \Vert_{{L^2}(\Omega)}$ are also bounded.  

Note that $\Vert \Delta \dr \Vert^2_{L^2(\Omega)}$ is bounded by the following inequality:
\begin{equation}
\int_\Omega \left|\Delta \dr \right|^2   \leqslant \int_\Omega 2\left(\Delta \dr+q^2 \dr \right)^2  +2\int_\Omega \left(q^2 \dr\right)^2.
\end{equation}
Given the boundedness of both $\Vert \dr \Vert_{L^2(\Omega)}$, $\Vert \Delta  \dr \Vert_{L^2(\Omega)}$, and \eqref{eq: Wp estimation}, we can establish the boundedness of $\Vert \dr \Vert_{W^{2,2}(\Omega)}$ which proves the coerciveness estimate for $\dr$. The weak lower semi-continuity is guaranteed in Lemma \ref{lemma: wsls} and the existence of a global minimizer of \eqref{eq:energy_LdG} follows immediately from the direct methods in the calculus of variations.
\end{proof}

We note that the energy density \( f_{angle} \) is not well-defined on the product space \( W_\Q \times W_{\dr} \), i.e., we cannot guarantee \( f_{angle} <\infty\) since it involves $(\text{Tr}(D^2\dr \Q))^2$, which requires higher regularity than \( W_{\dr} \) typically provides. If we modify the admissible space for $\dr$ to  $\bar{W}_{\dr} = W^{2,4}(\Omega) \cap \left\{\dr: \dr=\dr_{bc} \text{ on } \partial \Omega \right\}$, and the smectic layer energy density, $\lambda_1\left(\Delta\delta\rho + q^2\delta\rho\right)^2$ to $\lambda_1\left(\Delta\delta\rho + q^2\delta\rho\right)^4$, to control $\|\dr\|_{W^{2,4}(\Omega)}$, i.e., propose the following modified energy functional  
\begin{equation}
\begin{aligned}
\bar{F}(\Qvec,\delta\rho) =  &\int_{V}\bigg\{
\frac{\eta_1}{2}|\nabla\times\Qvec+2\sigma \Qvec|^2+\frac{\eta_2}{2}|\nabla\cdot\Qvec|^2+\frac{\eta_{24}}{2}(Q_{ij,k}Q_{ik,j}-Q_{ij,j}Q_{ik,k})\\
& + \frac{A}{2}\textrm{tr}\Qvec^2-\frac{B}{3}\textrm{tr}\Qvec^3+\frac{C}{4}(\textrm{tr}\Qvec^2)^2-f_{B,0} + \frac{d}{2}(\delta\rho)^2-\frac{e}{3}(\delta\rho)^3+\frac{f}{4}(\delta\rho)^4\\
&+ \lambda_1\left(\Delta\delta\rho + q^2\delta\rho\right)^4
 + \lambda_2 \left(\mathrm{tr}\left({D}^2\delta\rho \left(\mathbf{Q}+\frac{\Ivec_3}{3}\right)\right) + q^2\delta\rho\cos^2\theta_0\right)^2\bigg\} ,
\end{aligned}
 \label{modified energy functional}
\end{equation}
then the modified energy functional \eqref{modified energy functional} is well-defined on $\bar{W}_{\dr}$.

\begin{proposition}\label{existence proof}
Under the same conditions in Proposition \ref{existence in Q model}, the modified energy functional $\bar{F}(\Q,\dr)$ defined in \eqref{modified energy functional}, is coercive and w.l.s.c.~in the admissible space $W_\Q \times \bar{W}_{\dr}$, i.e., a global minimizer exists for the modified energy $\bar{F}(\Q,\dr)$ in $W_\Q \times \bar{W}_{\dr}$. 
\end{proposition}
\begin{proof}
Given $(\Q,\dr) \in W_\Q \times \bar{W}_{\dr}$, we have $\Q \in W^{1,2}_{\mathbf{S}_0}(\Omega) \hookrightarrow L^6_{\mathbf{S}_0}(\Omega)$, $\dr \in W^{2,4}(\Omega) \hookrightarrow L^\infty(\Omega)$, and ${D}^2\delta\rho \in L^4(\Omega)$, then
$$\Delta \dr +q^2 \dr \in L^4(\Omega), ~\mathrm{tr}\left({D}^2\dr \left(\Q+\I_3/3\right) \right) \in L^2(\Omega),$$
and thus, $\bar{F}(\Q,\dr)<\infty$. 

Since the energy density with respect to $\Q$ has not been modified in \eqref{modified energy functional}, we need only focus on the coerciveness and weak lower semi-continuity with respect to $\dr$ on $W^{2,4}(\Omega)$. Recall that since the bulk energy $f_{bs}(\dr)$ is a fourth order polynomial of $\dr$ with $f>0$ and $\int_\Omega f_{bs}(\dr)$ is bounded, we can deduce that $\Vert \dr \Vert_{{L^4}(\Omega)}$ are also bounded. Hence,
\begin{equation}
\Vert \Delta \dr \Vert_{L^4(\Omega)}\leqslant \Vert \Delta \dr+q^2 \dr \Vert_{L^4(\Omega)}+\Vert q^2 \dr \Vert_{L^4(\Omega)},
\end{equation}
is also bounded. Given the boundedness of $\Vert \dr \Vert_{L^4(\Omega)}$, $\Vert \Delta  \dr \Vert_{L^4(\Omega)}$ and \eqref{eq: Wp estimation} with $p=4$, we have established the boundedness of $\Vert \dr \Vert_{W^{2,4}(\Omega)}$ which guarantees the coerciveness estimate for $\dr$. The weak lower semi-continuity follows similarly as in Lemma \ref{lemma: wsls}.
\end{proof}
\begin{remark}
In practice, we are interested in energy minimizers and the energy functional can become infinite when \(\dr\) lacks sufficient regularity. However, the minimization process naturally avoids the low-regularity for \(\dr\). Meanwhile, noticing that the Euler--Lagrange equation of the modified energy \eqref{modified energy functional} is of eighth order, which can cause severe ill-conditioning problems in numerical computations, and hence, we use the original energy functional \eqref{eq:energy_LdG} in the remainder of this paper.
\end{remark}

\subsection{Oseen-Frank limit}
The nematic-smectic coupling energy in \eqref{eq:energy_LdG} is intuitively built with the uniaxial assumption on $\Q$ \cite{xia2021structural, xia2025cstar}. Although the minimizer is not always uniaxial in numerical computations, the bulk energy often dominates the total energy, which can lead to an almost uniaxial minimizer. In this subsection, we demonstrate that the uniaxiality assumption can be justified in a suitably defined asymptotic limit, for which the nematic bulk constants dominate all remaining material constants in an appropriately defined rescaled limit (or the so-called Oseen-Frank limit \cite{landau2010beta,liu2018oseen}) in the three-dimensional (3D) case. 

Note that the units of the parameters are $\sigma, q$ (m$^{-1}$); $\eta_i$ (N); $A, B, C, d, e, f$ (Nm$^{-2}$); and $\lambda_i$ (Nm$^2$). By rescaling the system according to $\bar{\x} = \frac{\x}{R}$, $\bar{F}=\frac{F
}{\eta_0 R}$, $\bar{q}=qR$, $\bar{\sigma}=\sigma R$,  $\bar{\lambda}_i=\frac{\lambda_i}{\eta_0 R^2}$, $\eta_i=\frac{\eta_i}{\eta_0}$, $\bar{d}=\frac{dR^2}{\eta_0}$, $\bar{e}=\frac{eR^2}{\eta_0}$, $\bar{f}=\frac{fR^2}{\eta_0}$, $\bar{A}=\frac{AR^2}{\eta_0}$, $\bar{B}=\frac{BR^2}{\eta_0}$,  $\bar{C}=\frac{CR^2}{\eta_0}$, $\bar{f}_{B,0}=\frac{f_{B,0}R^2}{\eta_0}$ with $R=1$m, $\eta_0=1$N, and drop all the bars, the non-dimensionalized energy is given by
\begin{equation}
\begin{aligned}
F(\Qvec,\delta\rho) =  &\int_{V}\bigg\{
\frac{\eta_1}{2}|\nabla\times\Qvec+2\sigma  \Qvec|^2+\frac{\eta_2}{2}|\nabla\cdot\Qvec|^2+\frac{\eta_{24}}{2}(Q_{ij,k}Q_{ik,j}-Q_{ij,j}Q_{ik,k})\\
& + \frac{A}{2}\textrm{tr}\Qvec^2-\frac{B}{3}\textrm{tr}\Qvec^3+\frac{C}{4}(\textrm{tr}\Qvec^2)^2-f_{B,0} +  \frac{d}{2}\delta\rho^2-\frac{e}{3}\delta\rho^3+\frac{f}{4}\delta\rho^4\\
&+ \lambda_1\left(\Delta\delta\rho + q^2\delta\rho\right)^2
 + \lambda_2 \left(\mathrm{tr}\left({D}^2\delta\rho \left(\mathbf{Q}+\frac{\Ivec_3}{3}\right)\right) + q^2\delta\rho\cos^2\theta_0\right)^2\bigg\},
\end{aligned}
 \label{non-dimensionalized}
\end{equation}
where all the parameters and the total energy are dimensionless. Note that we can also make different choices of $R$, so as to compare characteristic material-dependent length scales.

\begin{proposition}
[$\Gamma$-convergence]\label{gammar convergence} By taking a positive parameter $\lambda>0$, we rescale the elastic constants by \( A = \lambda A_0 \), \( B = \lambda B_0 \) and \( C = \lambda C_0 \) with $A_0<0$, and denote the resulting free energy functional \eqref{eq:energy_LdG} as \( F_\lambda \). Keeping other parameters fixed, as $\lambda \rightarrow \infty$, the following statements hold:\\
(a) For any family pair $(\Q_\lambda,\dr_\lambda)$, if there exists a positive constant $C_4$ such that $F_\lambda(\Q_\lambda,\dr_\lambda)\leqslant C_4,$ then $\Q_\lambda \rightharpoonup \Q^*$, at least on a subsequence, for some $\Q^* \in W^{1,2}_{\mathbf{S}_1}(\Omega)$, where 
\begin{equation}
\mathbf{S}_1=\{\Q: \Q=s_+(\n \otimes \n -\I_3/3)\}
\label{uniaxial manifold}
\end{equation}
is the uniaxial manifold.\\
(b)The family of energy functionals, $F_\lambda$, actually $\Gamma$-converges to $F_\infty$ in the weak topology of $W^{1,2}_{\mathbf{S}_0}(\Omega)$, 
i.e., 
\begin{equation}
\begin{aligned}
&{\forall} (\Q_\lambda,\dr_\lambda)\rightharpoonup(\Q,\dr) \in W_\Q\cap W^{1,2}_{\mathbf{S}_1}(\Omega) \times W_{\dr}, \ \liminf_{\lambda \rightarrow \infty}F_\lambda(\Q_\lambda,\dr_\lambda)\geqslant F_\infty(\Q,\dr),\\
&{\exists} (\Q_\lambda,\dr_\lambda)\rightharpoonup(\Q,\dr) \in W_\Q\cap W^{1,2}_{\mathbf{S}_1}(\Omega) \times W_{\dr}, \ \limsup_{\lambda \rightarrow \infty}F_\lambda(\Q_\lambda,\dr_\lambda)\leqslant F_\infty(\Q,\dr),
\end{aligned}
\label{liminf limsup}
\end{equation}
where $W_\Qvec$ and $W_{\dr}$ are defined in \eqref{eq: admissible Dirichlet for u}, $F_\infty$ is the Oseen--Frank limiting energy defined as 
\begin{equation}
\begin{aligned}
&F_\infty(\Qvec,\delta\rho)  = \int_{\Omega}\bigg\{
\frac{\eta_1}{2}|\nabla\times\Qvec+2\sigma \Qvec|^2+\frac{\eta_2}{2}|\nabla\cdot\Qvec|^2\\&+\frac{\eta_{24}}{2}(Q_{ij,k}Q_{ik,j}-Q_{ij,j}Q_{ik,k}) + \frac{d}{2}\delta\rho^2-\frac{e}{3}\delta\rho^3+\frac{f}{4}\delta\rho^4 \\
&+ \lambda_1\left(\Delta\delta\rho + q^2\delta\rho\right)^2
 + \lambda_2 \left(\mathrm{tr}\left({D}^2\delta\rho \left(\mathbf{Q}+\frac{\Ivec_3}{3}\right)\right) + q^2\delta\rho\cos^2\theta_0\right)^2\bigg\},
\end{aligned}
\label{OF limit energy original}
\end{equation}
with $\Q\in W^{1,2}_{\mathbf{S}_1}(\Omega)$ ($\mathbf{S}_1$ is the uniaxial manifold defined in \eqref{uniaxial manifold}).
\end{proposition}
\begin{proof}
By defining
\begin{equation*}
\bar{f}_{bn}(\Q):= \frac{A_0}{2}\tr \Q^2 - \frac{B_0}{3} \tr \Q^3 + \frac{C_0}{4}(\tr \Q^2)^2-f_{B,0},
\end{equation*}
we have $f_{bn}(\Q)=\lambda \bar{f}_{bn}(\Q)$.

(a) It is reasonable to assume that $\lambda>1$. Since the layer thickness energy and coupling energy are non-negative, from \eqref{eq: coercive with respect to Q} we have 
\begin{equation}
C_0\Vert \nabla \Q_\lambda\Vert_{L^{2}_{\mathbf{S}_0}(\Omega)}^2-\epsilon \Vert \nabla \Q_\lambda \Vert_{L^{2}_{\mathbf{S}_0}(\Omega)}^2-C_1(\epsilon) \Vert \Q_\lambda \Vert^2_{L_{\mathbf{S}_0}^2(\Omega)}  + \lambda \int_\Omega \bar{f}_{bn}(\Q_\lambda) \leqslant C_4,
\label{eq: boundness of sequence in (a)}
\end{equation}
where $\epsilon<C_0$. As the energy density $\bar{f}_{bn}(\Q)$ is a fourth-order polynomial of $\Q$ with positive leading term, we can infer the uniform boundedness of $\Vert \Q_\lambda \Vert_{W^{1,2}_{\mathbf{S}_0}(\Omega)}$ by the inequality \eqref{eq: boundness of sequence in (a)}. Then we have a subsequence, $\Q_\lambda \rightharpoonup \Q$ in $W_\Q$, and \eqref{eq: boundness of sequence in (a)} reduces to
\begin{equation}
\lambda \int_\Omega \bar{f}_{bn}(\Q_\lambda) \leqslant C_5 \  \text{and} \ \lim_{\lambda \rightarrow \infty}\int_\Omega \bar{f}_{bn}(\Q_\lambda) =0.
\label{bulk dissipate}
\end{equation}
By the w.l.s.c.~of $\int_\Omega \bar{f}_{bn}(\Q)$ and $\Q_\lambda \rightharpoonup \Q$ in $W_\Q$ together with \eqref{bulk dissipate}, we deduce that $\Q$ is uniaxial almost everywhere, with constant order parameter $s_+$, i.e., is a minimizer of $f_{bn}$ for $A_0<0$. \\ 
(b) The liminf condition in the first line of \eqref{liminf limsup} can be directly obtained by 
\begin{equation}
\liminf_{\lambda \rightarrow \infty}F_\lambda(\Q_\lambda,\dr_\lambda) \geqslant \liminf_{\lambda \rightarrow \infty}F_\infty(\Q_\lambda,\dr_\lambda)\geqslant F_\infty (\Q,\dr),
\end{equation}
where the first inequality stems from the non-negative bulk energy density $\bar{f}_{bn}(\Q)$, and the second inequality follows from the weak lower semicontinuity of $F_\infty$.

For the limsup condition in the second line of \eqref{liminf limsup}, we can take the sequence $(\Q_\lambda,\dr_\lambda)=(\Q,\dr)$.  
Then 
\begin{equation*}
\limsup_{\lambda \rightarrow \infty}F_\lambda(\Q_\lambda,\dr_\lambda) = \limsup_{\lambda \rightarrow \infty}F_\lambda(\Q,\dr)=F_\infty(\Q,\dr),
\end{equation*}
where $\Q \in W_\Q\cap W^{1,2}_{\mathbf{S}_1}(\Omega)$ by choice, which ensures that the nematic bulk energy $\bar{f}_{bn}(\Q)$ vanishes.
\end{proof}

We now obtain a direct corollary from the $\Gamma$-convergence in Proposition \ref{gammar convergence}.

\begin{corollary}\label{corollary}
Let $(\Q_\lambda^*,\dr_\lambda^*)$ be a global minimizer of $F_\lambda$ in $W_\Q \times W_{\dr}$. Then $(\Q_\lambda^*,\dr_\lambda^*)\rightharpoonup (\Q_\infty^*,\dr_\infty^*)$, $F_\lambda(\Q_\lambda^*,\dr_\lambda^*)\rightarrow F_\infty(\Q^*_\infty,\dr^*_\infty)$, where $(\Q^*_\infty,\dr^*_\infty)$ is the global minimizer of $F_\infty$ defined in \eqref{OF limit energy original}. 
\end{corollary}
\begin{proof}
We fix a configuration $(\bar{\Q},\bar{\dr})\in W_\Q \cap W^{1,2}_{\mathbf{S}_1}(\Omega) \times W_\dr$, whose nematic bulk energy vanishes, and hence, for all $\lambda$,
\begin{equation}\label{uniform boundness}
F_\lambda(\Q_\lambda^*,\dr_\lambda^*)\leqslant F_\lambda(\bar{\Q},\bar{\dr})=F_\infty(\bar{\Q},\bar{\dr}),
\end{equation}
where the first inequality holds because $(\Q_\lambda^*,\dr_\lambda^*)$ is the global minimizer of $F_\lambda$.
Actually, this inequality represents that $F_\lambda(\Q_\lambda^*,\dr_\lambda^*)$ has a uniform upper bound with respect to $\lambda$. 
By Proposition \ref{gammar convergence}(a) and \eqref{uniform boundness}, there exists a limit configuration $(\Q_\lambda^*,\dr_\lambda^*)\rightharpoonup (\Q_\infty^*,\dr_\infty^*)$, where $\Q_\infty^* \in W^{1,2}_{\mathbf{S}_1}(\Omega)$, and we deduce

\begin{align}
\limsup_{\lambda\rightarrow \infty}F_\lambda(\Q_\lambda^*,\dr_\lambda^*)\leqslant \limsup_{\lambda\rightarrow \infty}F_\lambda(\Q_\infty^*,\dr_\infty^*)=F_\infty(\Q_\infty^*,\dr_\infty^*),
\label{eq: sup term}
\\
\liminf_{\lambda\rightarrow \infty}F_\lambda(\Q_\lambda^*,\dr_\lambda^*)\geqslant \liminf_{\lambda\rightarrow \infty} F_\infty(\Q_\lambda^*,\dr_\lambda^*) \geqslant F_\infty(\Q_\infty^*,\dr_\infty^*),
\label{eq: inf term}
\end{align}
where the inequality in \eqref{eq: sup term} follows from that $(\Q_\lambda^*,\dr_\lambda^*)$ is the global minimizer of $F_\lambda$, the equality in \eqref{eq: sup term} follows from that $\Q_\infty^*$ is uniaxial. The first inequality in \eqref{eq: inf term} follows from the fact that the bulk energy $f_{bn}$ is non-negative, while the last inequality follows from the weak lower semi-continuity, which leads to $F_\lambda(\Q_\lambda^*,\dr_\lambda^*)\rightarrow F_\infty(\Q^*_\infty,\dr^*_\infty)$.

For any $(\Q, \dr) \in W_\Q \cap W^{1,2}_{\mathbf{S}_1}(\Omega)\times W_\dr$, we take the sequence guaranteed by Proposition \ref{gammar convergence}(b) to get
\begin{equation*}
\begin{aligned}
F_\infty(\Q, \dr)\geqslant \limsup_{\lambda \rightarrow \infty}F_\lambda(\Q_\lambda,\dr_\lambda) &\geqslant  \limsup_{\lambda \rightarrow \infty}F_\lambda(\Q_\lambda^*,\dr_\lambda^*)\\
&\geqslant \liminf_{\lambda \rightarrow \infty}F_\lambda(\Q_\lambda^*,\dr_\lambda^*) \geqslant F_\infty(\Q_\infty^*, \dr_\infty^*),
\end{aligned}
\end{equation*}
which implies that $(\Q_\infty^*, \dr_\infty^*)$ is a global minimizer of $F_\infty$. Here, the second inequality comes from the assumption that $(\Q_\lambda^*,\dr_\lambda^*)$  is a minimizer of $F_\lambda$, and the last inequality utilizes Proposition \ref{gammar convergence}(b).
\end{proof}

So far, we have emphasized the weak convergence of global energy minimizers to a uniaxial limiting map, belonging to the vacuum manifold of the nematic bulk energy, as \(\lambda\) approaches infinity. In the following Proposition, we demonstrate that there exists a subsequence of this weakly convergent sequence that can strongly converge to the uniaxial limiting solution.

\begin{proposition}[Strong convergence]
\label{proposition: strong convergence} Let $\Omega \in \R^3$ be a simply-connected bounded open set with $C^{1,1}$ boundary. Let $(\Q_\lambda^*,\dr_\lambda^*)$ be a global minimizer of $F_\lambda$ in the admissible space $W_\Q \times W_{\dr}$. There exists a sequence $\lambda_k \rightarrow \infty$ such that $(\Q_{\lambda_k}^*, \dr_{\lambda_k}^*) \rightarrow (\Q_\infty^*,\dr_\infty^*)$ strongly in $W_\Q \times W_\dr$, where $(\Q_\infty^*,\dr_\infty^*)$ is a global minimizer of $F_\infty$.
\end{proposition}
\begin{proof}
By the w.l.s.c. property stated in Lemma \ref{lemma: wsls}, we deduce that
\begin{equation}\label{eq: limsup other density}
\begin{aligned}
&\liminf_{\lambda \rightarrow \infty} \int_\Omega \left\{f_{layer}(\dr_\lambda^*)+f_{angle}(\Q_\lambda^*,\dr_\lambda^*)+f_{bs}(\dr_\lambda^*) \right\} \\
&\geqslant \int_\Omega  \left\{f_{layer}(\dr_\infty^*)+f_{angle}(\Q_\infty^*,\dr_\infty^*)+f_{bs}(\dr_\infty^*) \right\},\\ 
&\liminf_{\lambda \rightarrow \infty} \int_\Omega \left\{f_{el}(\Q_\lambda^*)+f_{angle}(\Q_\lambda^*,\dr_\lambda^*)+f_{bs}(\dr_\lambda^*) \right\} \\
&\geqslant \int_\Omega  \left\{f_{el}(\Q_\infty^*)+f_{angle}(\Q_\infty^*,\dr_\infty^*)+f_{bs}(\dr_\infty^*) \right\},\\ 
&\liminf_{\lambda \rightarrow \infty} \int_\Omega \lambda \bar{f}_{bn}(\Q_\lambda^*)\geqslant 0. 
\end{aligned}
\end{equation}
By Corollary \ref{corollary}, we have $\lim_{\lambda \rightarrow \infty} F_\lambda(\Q_\lambda^*,\dr_\lambda^*)=F_\infty(\Q^*_\infty,\dr^*_\infty)$, i.e.,
\begin{equation}
\begin{aligned}
&\lim_{\lambda \rightarrow \infty} \int_\Omega \left\{f_{el}(\Q_\lambda^*)+\lambda \bar{f}_{bn}(\Q_\lambda^*)+f_{bs}(\dr_\lambda^*)+f_{layer}(\dr_\lambda^*)+f_{angle}(\Q_\lambda^*,\dr_\lambda^*) \right\} \\
&= \int_\Omega  \left\{f_{el}(\Q_\infty^*)+f_{bs}(\dr_\infty^*)+f_{layer}(\dr_\infty^*)+f_{angle}(\Q_\infty^*,\dr_\infty^*) \right\}.
\end{aligned}
\label{eq: limsup total energy}
\end{equation}
Together with \eqref{eq: limsup other density}, we have
\begin{equation}
\begin{aligned}
&\limsup_{\lambda \rightarrow \infty} \int_\Omega f_{el}(\Q^*_\lambda) &\leqslant \int_\Omega f_{el}(\Q^*_\infty),
\end{aligned}
\label{lower}
\end{equation}
and $\limsup_{\lambda \rightarrow \infty} \int_\Omega f_{layer}(\dr^*_\lambda) \leqslant \int_\Omega f_{layer}(\dr^*_\infty)$
\begin{footnote}{Note that $\limsup x_\lambda-y_\lambda \leqslant \limsup x_\lambda-\liminf y_\lambda$, we get
\begin{equation}
\begin{aligned}
&\limsup_{\lambda \rightarrow \infty} \int_\Omega f_{el}(\Q^*_\lambda) \\
&\leqslant \limsup_{\lambda \rightarrow \infty} \int_\Omega \left\{f_{el}(\Q_\lambda^*)+\lambda \bar{f}_{bn}(\Q_\lambda^*)+f_{bs}(\dr_\lambda^*)+f_{layer}(\dr_\lambda^*)+f_{angle}(\Q_\lambda^*,\dr_\lambda^*) \right\}\\ 
&\quad -\liminf_{\lambda \rightarrow \infty}  \int_\Omega \left\{f_{layer}(\dr_\lambda^*)+f_{angle}(\Q_\lambda^*,\dr_\lambda^*)+f_{bs}(\dr_\lambda^*) \right\}-\liminf_{\lambda \rightarrow \infty} \int_\Omega \lambda \bar{f}_{bn}(\Q_\lambda^*)\\
&\leqslant \int_\Omega f_{el}(\Q^*_\infty),
\end{aligned}
\end{equation}
and similarly $\limsup_{\lambda \rightarrow \infty} \int_\Omega f_{layer}(\dr^*_\lambda) \leqslant \int_\Omega f_{layer}(\dr^*_\infty)$.
} \end{footnote}.
By the w.l.s.c.~property of $\int_\Omega f_{el}(\Q^*_\lambda)$ and $\int_\Omega f_{layer}(\dr^*_\lambda)$, we have the following: 
\begin{equation}
\begin{aligned}
\liminf_{\lambda \rightarrow \infty} \int_\Omega f_{el}(\Q^*_\lambda) &\geqslant \int_\Omega f_{el}(\Q^*_\infty),\\
\liminf_{\lambda \rightarrow \infty} \int_\Omega f_{layer}(\dr^*_\lambda) &\geqslant \int_\Omega f_{layer}(\dr^*_\infty).
\end{aligned}
\label{upper}
\end{equation}
Together with \eqref{lower} and \eqref{upper}, we have the following equalities:
\begin{equation}
\begin{aligned}
\lim_{\lambda \rightarrow \infty} \int_\Omega f_{el}(\Q^*_\lambda) &= \int_\Omega f_{el}(\Q^*_\infty), \\
\lim_{\lambda \rightarrow \infty} \int_\Omega f_{layer}(\dr^*_\lambda) &= \int_\Omega f_{layer}(\dr^*_\infty).
\end{aligned}
\label{equal}
\end{equation}
Recalling that
\begin{equation*}
\begin{aligned}
\int_\Omega f_{el}(\Q) =\int_\Omega & \frac{\eta_1}{2}|\nabla \Q|^2+\frac{\eta_2-\eta_{24}}{2}Q_{ij,j}Q_{ik,k}+\frac{\eta_{24}-\eta_1}{2}Q_{ij,k}Q_{ik,j} \\&+2\eta_1\sigma\epsilon_{ikl}Q_{lj,k}Q_{ij}+2\eta_1\sigma^2|\Q|^2,
\end{aligned}
\end{equation*}
we define 
\begin{equation}
\int_\Omega  \frac{\eta_1}{2}|\nabla \Q|^2+\frac{\eta_2-\eta_{24}}{2}Q_{ij,j}Q_{ik,k}+\frac{\eta_{24}-\eta_1}{2}Q_{ij,k}Q_{ik,j} =\Vert \P \nabla \Q \Vert^2_{L^2_{\mathbf{S}_0}(\Omega)},
\end{equation}
where $\P$ is a positive definite matrix. By \eqref{equal} and the fact that $\Q_\lambda^* \rightarrow \Q_\infty^*$, $\nabla \Q_\lambda^* \rightharpoonup \nabla \Q_\infty^*$, we obtain that
\begin{equation}
\Vert \P \nabla \Q_\lambda^* \Vert^2_{L^2_{\mathbf{S}_0}(\Omega)} \rightarrow \Vert \P \nabla \Q_\infty^* \Vert^2_{L^2_{\mathbf{S}_0}(\Omega)}.
\end{equation}
Together with $\Q_\lambda^* \rightharpoonup \Q_\infty^*$, we have 
\begin{equation}
\begin{aligned}
&\lim_{\lambda \rightarrow \infty}\Vert \P (\nabla \Q_\lambda^*-\nabla \Q_\infty^*) \Vert^2_{L^2_{\mathbf{S}_0}(\Omega)}
\\
&=\lim_{\lambda \rightarrow \infty}\Vert \P \nabla \Q_\lambda^* \Vert^2_{L^2_{\mathbf{S}_0}(\Omega)}+\Vert \P \nabla \Q_\infty^* \Vert^2_{L^2_{\mathbf{S}_0}(\Omega)}-\lim_{\lambda \rightarrow \infty}2\int_\Omega \langle \P \nabla \Q_\lambda^*,\P \nabla \Q_\infty^* \rangle=0.
\label{strong convergence}
\end{aligned}
\end{equation}
That is $\nabla \Q_\lambda^* \rightarrow \nabla \Q_\infty^*\text{ in } L^2_{\mathbf{S}_0}(\Omega)$.

By the compact embedding of $W^{1,2}(\Omega)$ into $L^2(\Omega)$, one obtains $ \Q_\lambda^* \rightarrow  \Q_\infty^*\text{ in } L^2_{\mathbf{S}_0}(\Omega)$. Combined with \eqref{strong convergence}, there exists a subsequence of global energy minimizers, $\Q_\lambda^*$, that strongly converges to $\Q_\infty^*$ as $\lambda\to\infty$.

By Corollary \ref{corollary}, we have $\Delta \dr_\lambda^* \rightharpoonup \Delta \dr_\infty^*$ in $L^2(\Omega)$, $\dr_\lambda^* \rightarrow \dr_\infty^*$ strongly in $L^4(\Omega)$ from $\dr_\lambda^* \rightharpoonup  \dr_\infty^*$ in $W^{2,2}(\Omega)$, and $\lim_{\lambda \rightarrow \infty} \int_\Omega f_{layer}(\dr^*_\lambda) = \int_\Omega f_{layer}(\dr^*_\infty)$, i.e., $\Delta \dr_\lambda^*+q^2 \dr_\lambda^*\rightharpoonup \Delta\dr_\infty^* + q^2\dr_\infty^*$ in $L^2(\Omega)$, and $\|\Delta \dr_\lambda^*+q^2 \dr_\lambda^*\|\rightarrow \|\Delta\dr_\infty^* + q^2\dr_\infty^*\|$, then we have that 
\begin{equation}
\Vert\Delta \dr_\lambda^* + q^2\dr_\lambda^*- (\Delta\dr_\infty^* + q^2\dr_\infty^*)\Vert_{L_2(\Omega)}\rightarrow 0.
\end{equation}
Hence,
\begin{equation*}
\Vert \Delta \dr_\lambda^* -\Delta \dr_\infty^* \Vert_{L_2(\Omega)} \rightarrow 0,
\end{equation*}
which, together with \eqref{eq: Wp estimation}, leads to
\begin{equation}
\Vert {D}^2 \dr_\lambda^* - {D}^2 \dr_\infty^* \Vert_{L^2(\Omega)} \rightarrow 0.
\end{equation}
Hence, $\dr_\lambda^* \rightarrow \dr_\infty^*$ strongly in $H^2(\Omega)$. 
\end{proof}

\begin{remark}
From Proposition \ref{proposition: strong convergence}, one can study the qualitative properties of minimizers of $F_\lambda$ in terms of the minimizers of $F_\infty$, for large enough $\lambda$. In the subsequent sections, we show that $F_\infty$ exhibits a nematic-smectic phase transition at $T^*$, i.e., its global minimizer is nematic for $T>T^*$ and smectic for $T<T^*$. For sufficiently large $\lambda$, Proposition \ref{proposition: strong convergence} ensures that the minimizer of $F_\lambda$ is sufficiently close to that of $F_\infty$ and hence, one can expect that $F_\lambda$ undergoes a similar nematic-smectic phase transition for sufficiently large $\lambda$.
\end{remark}

We study the qualitative properties of minimizers of $F_\infty$ in the remainder of the paper.

\subsection{Helical configuration in the Oseen--Frank limit}
In the previous section, we demonstrate that the global minimizer of \eqref{eq:energy_LdG} with $A<0$ and $|A|, B, C \gg |d|, e,f,\lambda_1, \lambda_2$ can be approximated (to leading order) by the uniaxial limiting map, $\Q=s_+(\n \otimes \n - \I_3/3)$ with $\n=[n_1,n_2,n_3]^\top \in \mathbb{S}^2$. We substitute this limiting form into the Oseen--Frank limiting energy, $F_\infty(\Q,\dr)$ \eqref{OF limit energy original}, to obtain the following functional of $(\n,\dr)$:
\begin{equation}\label{eq:energy-3d vector}
\begin{aligned}
F(\n,\dr)=& \int_\Omega \bigg\{k_1(\nabla \cdot \n)^2+k_2|\n \cdot (\nabla \times \n)+\sigma|^2 +k_3|\n \times (\nabla \times \n)|^2\\&+(k_2+k_4)(\textbf{tr}(\nabla \n)^2-(\nabla \cdot \n)^2)
+\frac{d}{2}\delta\rho^2+\frac{f}{4}\delta\rho^4\\
&+ \lambda_1\left(\Delta\delta\rho + q^2\delta\rho\right)^2
 + \lambda_2 \left(\mathrm{tr}\left({D}^2\delta\rho \left(\n \otimes \n\right)\right) + q^2\delta\rho\cos^2\theta_0\right)^2\bigg\},
\end{aligned}
\end{equation}
where 
\begin{equation}\label{relationship between Q and v parameters}
k_1=k_3=\frac{s_+^2(\eta_1 +\eta_2)}{2},~k_2=s_+^2\eta_1,~k_4=\frac{s_+^2(\eta_{24}-\eta_1)}{2},
\end{equation}
and $s_+$ is a fixed constant defined in \eqref{eq: s+}. 
We restrict the Oseen–Frank constants to ensure compatibility with \eqref{relationship between Q and v parameters} and \eqref{eq: positive quadratic nabla Q}: 
\begin{equation}\label{assumption oseen frank}
k_1=k_3 \geqslant k_2> k_2+k_4 \geqslant C_6  >0 > k_4.
\end{equation}
In the remainder of this subsection, we fix the working domain to be $\Omega=\Omega_{2D}\times [0,h]$ for simplicity, where $\Omega_{2D}$ is a convex simply-connected bounded set with a $C^{1,1}$ smooth boundary. Since it is a convex cylindrical region (hence a domain with 2-codimensional edges \cite{dauge2006elliptic}), the regularity estimates in \eqref{eq: Wp estimation} remain valid \cite{dauge2006elliptic,grisvard2011elliptic}. 
We study the minimizers of the Oseen--Frank limiting energy $F(\n,\dr)$, where the director $\n$ is tangent on the top and bottom plates, $z=0$ and $z=h$, and free on the lateral surfaces. Thus, the admissible space for $\n$ is 
\begin{equation}
W_{\n}=\{\n \in W^{1,2}(\Omega): |\n|\equiv 1, n_3(x,y,z=0\text{~and~}h)=0 \text{ for } (x,y)\in \Omega_{2D} \}.
\end{equation}

\begin{proposition}\label{existence in vector model}
Let the elastic constants in \eqref{eq:energy-3d vector} satisfy the assumption \eqref{assumption oseen frank}. The Oseen--Frank limiting energy \eqref{eq:energy-3d vector} has a global minimizer $(\Tilde{\n},\Tilde{\dr})$ in the admissible space $W_{\n}\times W_\dr$. 
\end{proposition}
\begin{proof}
From the proof of Proposition \ref{existence in Q model},  it is sufficient to verify the coercivity of $F(\n,\dr)$ with respect to $\n$. Note that $$\textbf{tr}(\nabla \n)^2= |\nabla \n|^2-|\nabla \times \n|^2=|\nabla \n|^2-|\n\cdot(\nabla \times \n)|^2-|\n\times(\nabla \times \n)|^2,$$ we have
\begin{equation*}
\begin{aligned}
F(\n,\dr)\geqslant & \int_\Omega \bigg\{k_1(\nabla \cdot \n)^2+k_2|\n \cdot (\nabla \times \n)+\sigma|^2 +k_3|\n \times (\nabla \times \n)|^2\\&+(k_2+k_4)(\textbf{tr}(\nabla \n)^2-(\nabla \cdot \n)^2) \bigg\}- \hat{C} \\
&=\int_\Omega \bigg\{(k_1-k_2-k_4)(\nabla \cdot \n)^2 + (k_3-k_2-k_4)|\n \times (\nabla \times \n)|^2\\
&-k_4(\n \cdot (\nabla \times \n))^2 + (k_2+k_4)|\nabla \n|^2 +2\sigma k_2 \n\cdot (\nabla \times \n)\bigg\} - \hat{C}\\
&\geqslant \int_\Omega \left\{ (k_2+k_4)|\nabla \n|^2 + 2\sigma k_2 \n \cdot (\nabla \times \n)\right\} -\hat{C}\\
&\geqslant \int_\Omega \left\{(k_2+k_4)|\nabla \n|^2-\epsilon (\n \cdot (\nabla \times \n))^2 \right\}- \hat{C}(\epsilon)\\
&\geqslant (C_6-2\epsilon)\Vert \nabla \n \Vert_{L^2(\Omega)}^2-\hat{C}(\epsilon),
\end{aligned}
\end{equation*}
where the second-to-last inequality follows from Young’s inequality, for sufficiently small and positive $\epsilon$. 
\end{proof}
In the Oseen--Frank model for cholesteric liquid crystals \cite{de1993physics}, the free energy of the helical configuration 
\begin{equation}\label{helical configuration}
\n_{\sigma}(\mathbf{x}) = [\cos(\sigma z), \sin(\sigma z), 0]^\top
\end{equation}
actually vanishes. In the remainder of this section, we  demonstrate that if the re-scaled Oseen--Frank constants in \eqref{eq:energy-3d vector} are sufficiently large, then the energy-minimizing vector field $\Tilde{\n}$ can be well approximated by the helical profile, \(\n_{\sigma}\) (although $\Tilde{\dr}$ may be non-zero).

\begin{lemma}\label{Lemma 4.2}
Assume that the Oseen-Frank elastic constants satisfy the relation \eqref{assumption oseen frank}. If $(\Tilde{\n},\Tilde{\dr})$ is a global minimizer of $F(\n,\dr)$ in $W_\n \times W_\dr$, then 
\begin{equation}\label{eq: boundedness of nabla n}
\Vert 
\nabla \Tilde{\n} \Vert^2_{L^2(\Omega)} 
 \leqslant \left(\frac{d^2}{4fC_6}-\frac{k_2 \sigma^2} {k_4}\right)|\Omega|.
\end{equation}
Moreover, if there exists a constant $C_7$ such that $k_2+k_4 \leqslant C_7$ and $k_2\geqslant 4C_7$, then
\begin{equation}
\Vert \nabla \times \Tilde{\n} +\sigma  \Tilde{\n}\Vert_{L^2(\Omega)}^2 \leqslant \frac{d^2|\Omega|}{2f k_2}+\frac{4C_7\sigma^2|\Omega|}{k_2},
\label{vanishment of the twist term}
\end{equation}
where $d,f$ are the re-scaled bulk constants in \eqref{eq:energy-3d vector}.
\end{lemma}
\begin{proof}
Since $(\Tilde{\n},\Tilde{\dr})$ is a global minimizer of $F(\n,\dr)$ in $W_\n \times W_\dr$, we have $F(\n_\sigma,\dr\equiv 0)\geqslant F(\Tilde{\n},\Tilde{\dr})$, i.e., 
\begin{equation*}
\begin{aligned}
F(\n_\sigma,\dr\equiv 0)&=0\\
&\geqslant \int_\Omega \bigg\{k_1(\nabla \cdot \Tilde{\n})^2+k_2|\Tilde{\n} \cdot (\nabla \times \Tilde{\n})+\sigma|^2 +k_3|\Tilde{\n} \times (\nabla \times \Tilde{\n})|^2\\
&+(k_2+k_4)(\textbf{tr}(\nabla \Tilde{\n})^2-(\nabla \cdot \Tilde{\n})^2)
+\frac{d}{2}\Tilde{\dr}^2+\frac{f}{4}\Tilde{\dr}^4\\
&+ \lambda_1\left(\Delta \Tilde{\dr} + q^2\Tilde{\dr}\right)^2
 + \lambda_2 \left(\mathrm{tr}\left(\mathcal{D}^2\Tilde{\dr} \left(\Tilde{\n} \otimes \Tilde{\n} \right)\right) + q^2\Tilde{\dr}\cos^2\theta_0\right)^2 \bigg\} \\
 & \geqslant \int_\Omega k_1(\nabla \cdot \Tilde{\n})^2+k_2|\Tilde{\n} \cdot (\nabla \times \Tilde{\n})+\sigma|^2 +k_3|\Tilde{\n} \times (\nabla \times \Tilde{\n})|^2\\
&+(k_2+k_4)(\textbf{tr}(\nabla \Tilde{\n})^2-(\nabla \cdot \Tilde{\n})^2)
-\frac{d^2}{4f}|\Omega|,
\end{aligned}
\end{equation*}
and note that $|\Tilde{\n} \cdot (\nabla \times \Tilde{\n})+\sigma|^2=|\nabla \times \Tilde{\n} + \sigma \Tilde{\n}|^2-|\Tilde{\n} \times (\nabla \times \Tilde{\n})|^2$, $\textbf{tr}(\nabla \Tilde{\n})^2=|\nabla \Tilde{\n}|^2-|\nabla \times \Tilde{\n}|^2$. Then,
\begin{equation}\label{inequality for bounding nabla n and twist n}
\begin{aligned}
\frac{d^2}{4f}|\Omega|&\geqslant \int_\Omega \bigg\{(k_1-k_2-k_4)(\nabla \cdot \Tilde{\n})^2 + (k_3-k_2)|\Tilde{\n} \times (\nabla \times \Tilde{\n})|^2 \\
&\qquad+k_2|\nabla \times \Tilde{\n} + \sigma \Tilde{\n}|^2-(k_2+k_4)|\nabla \times \Tilde{\n}|^2+(k_2+k_4) |\nabla \Tilde{\n}|^2 \bigg\}\\
&\geqslant \int_\Omega \left\{k_2|\nabla \times \Tilde{\n}+\sigma \Tilde{\n}|^2+(k_2+k_4)|\nabla \Tilde{\n}|^2-(k_2+k_4)|\nabla \times \Tilde{\n}|^2 \right\}\\
&\geqslant  \int_\Omega \left\{(k_2+k_4)|\nabla \times \Tilde{\n}|^2+\frac{k_2(k_2+k_4)}{k_4}\sigma^2+(k_2+k_4)|\nabla \Tilde{\n}|^2-(k_2+k_4)|\nabla \times \Tilde{\n}|^2 \right\},
\end{aligned}
\end{equation}
where the last inequality follows from Young’s inequality. Therefore,
\begin{equation}\label{boundness of the nabla n}
\left(\frac{d^2}{4fC_6}-\frac{k_2 \sigma^2} {k_4}\right)|\Omega|\geqslant \left(\frac{d^2}{4f(k_2+k_4)}-\frac{k_2 \sigma^2} {k_4}\right)|\Omega|\geqslant \Vert 
\nabla \Tilde{\n} \Vert^2_{L^2(\Omega)}.
\end{equation}
Moreover, if $k_2+k_4 \leqslant C_7$ and $k_2\geqslant 4C_7$, then
\begin{equation*}
\begin{aligned}
&\frac{d^2}{4f}|\Omega|\geqslant  \int_\Omega \left\{k_2|\nabla \times \Tilde{\n}+\sigma \Tilde{\n}|^2+(k_2+k_4)|\nabla \Tilde{\n}|^2-(k_2+k_4)|\nabla \times \Tilde{\n}|^2 \right\}\\
&\geqslant  \int_\Omega \left\{ \frac{k_2}{2}|\nabla \times \Tilde{\n}+\sigma \Tilde{\n}|^2+\frac{k_2}{2}|\nabla \times \Tilde{\n}+\sigma \Tilde{\n}|^2-(k_2+k_4)|\nabla \times \Tilde{\n}|^2 \right\}\\
&\geqslant \int_\Omega \left\{ \frac{k_2}{2}|\nabla \times \Tilde{\n}+\sigma \Tilde{\n}|^2+2C_7|\nabla \times \Tilde{\n}+\sigma \Tilde{\n}|^2-C_7|\nabla \times \Tilde{\n}|^2 \right\}\\
&\geqslant \int_\Omega \left\{ \frac{k_2}{2}|\nabla \times \Tilde{\n}+\sigma \Tilde{\n}|^2-2C_7\sigma^2 \right\},
\end{aligned}
\end{equation*}
which leads to \eqref{vanishment of the twist term}.
\end{proof}

Subsequently, from \eqref{vanishment of the twist term} and the assumptions of Lemma \ref{Lemma 4.2}, the twist energy density $\Vert \nabla \times \Tilde{\n} +\sigma  \Tilde{\n}\Vert_{L^2(\Omega)}^2$ of the minimizers tends to zero as \( k_2 \rightarrow \infty \) (recall that all parameters, including $k_2$, are nondimensional due to the scaling in \eqref{non-dimensionalized}), and thus their \( W^{1,2} \) norms are uniformly bounded. Consequently, one obtains a limiting helical configuration as $k_2\to \infty$. 

\begin{corollary}\label{limit to helical configuration as k2 to inf}
Let \(\{(\mathbf{n}^*_i, \dr^*_i)\}\) be a sequence of minimizers for \(F(\n,\dr)\) with rescaled Oseen-Frank constants \((k_j^i,j=1,2,3,4\) satisfying the relation \eqref{assumption oseen frank} and assumption in Lemma \ref{Lemma 4.2}, such that \(\lim_{i \to \infty} k_2^i = \infty\). There exists a subsequence \(\{(\mathbf{n}^*_{i_k},\dr^*_{i_k})\}\) and function \(\mathbf{n}^*_\infty \in W_\n \) such that \(\mathbf{n}^*_{i_k} \rightharpoonup \mathbf{n}^*_\infty\) in \(W^{1,2}(\Omega)\), as \(i_k \to \infty\), where \(\mathbf{n}^*_\infty\) satisfies
\[
\nabla \times \mathbf{n}^*_\infty + \sigma \mathbf{n}^*_\infty = 0 \quad \text{in~} \Omega.
\]
\end{corollary}
The following Lemma, which has been proven in \cite{ericksen1967general}, demonstrates that $\mathbf{n}^*_\infty$ is the helical configuration (up to a rotation), and this conclusion holds for all $d$.
\begin{lemma}\label{configuration satisfies pde for helical configuration}
Consider \(\mathbf{n} \in W_\n\) such that
\begin{equation}\label{pde for helical configuration}
\nabla \times \mathbf{n} + \sigma \mathbf{n} = 0 \quad \text{in~} \ \Omega.
\end{equation}
Then
\begin{equation}\label{Rotation matrix}
\mathbf{n}(\mathbf{x}) = \mathbf{R} \mathbf{n}_{\sigma}(\mathbf{R}^T \mathbf{x}), \text{~with~} \mathbf{R}=\begin{bmatrix}
\cos \theta & -\sin \theta & 0 \\
\sin \theta & \cos \theta & 0 \\
0 & 0 & 1
\end{bmatrix}.
\end{equation}
\end{lemma}

\begin{proposition}\label{helical configuration limit}
Let \(\{(\mathbf{n}^*_i, \dr^*_i)\}\) be a sequence of minimizers for \(F_i(\n,\dr)\), where \(F_i(\n,\dr)\) is the total energy $F$ with Frank constants \((k_j^i,j=1,2,3,4)\) satisfying the relation \eqref{assumption oseen frank} and assumption in Lemma \ref{Lemma 4.2}, and such that \(\lim_{i \to \infty} k_2^i = \infty\). There exists a limit $(\n_\infty^*, \dr_\infty^*)$ and a subsequence $(\n_{i_k}^*, \dr_{i_k}^*)$ such that
\begin{equation}
\n_{i_k}^*\rightarrow \n_\infty^* \ in \ L^4(\Omega, \mathbb{S}^2), ~ \dr_{i_k}^* \rightharpoonup \dr_\infty^* \ in \ H^2(\Omega).
\end{equation}
Moreover, $\n_\infty^*$ is the helical configuration up to a rotation as defined in \eqref{Rotation matrix}, and $\lim_{k\rightarrow \infty}F_{i_k}(\n^*_{i_k},\dr^*_{i_k})=F_\infty(\n^*_\infty,\dr^*_\infty),$
where
\begin{equation}\label{eq: F_infty}
\begin{aligned}
F_\infty(\n,\dr)=& \int_\Omega \bigg\{ \frac{d}{2}\delta\rho^2+\frac{f}{4}\delta\rho^4\\
&+ \lambda_1\left(\Delta\delta\rho + q^2\delta\rho\right)^2
 + \lambda_2 \left(\mathrm{tr}\left({D}^2\delta\rho \left(\n \otimes \n\right)\right) + q^2\delta\rho\cos^2\theta_0\right)^2\bigg\}.
\end{aligned}
\end{equation}

\end{proposition}
\begin{proof}
The proof is similar to that of \cite[Lemma 4]{bauman2002phase} and \cite[Theorem 3.1]{joo_phase_2006}. By Proposition \ref{limit to helical configuration as k2 to inf}, the weak limit of the minimizers of $F(\n, \dr)$ satisfies \eqref{helical configuration} and is a helical configuration as described in Lemma \ref{configuration satisfies pde for helical configuration}. Specifically, for any helical configuration with a uniform density, we have $$F_i(\n_{\sigma},\dr\equiv 0)=0,$$ which serves as an upper bound for the energy of the minimizers. Thus, 
\begin{equation}
F_{i}(\n_{i}^*, \dr_{i}^*)\leqslant F_{i}(\n_{\sigma},\dr\equiv 0)=0.
\label{boundness of energy when k limits to infinity}
\end{equation}
From the condition $\lim_{i\rightarrow\infty}k_2^i= \infty$, $k_4<0$, $k_2+k_4 \leqslant C_7$, we have that $\lim_{i\rightarrow\infty}\left|k_2^i/k_4^i\right|=1$, and hence, we can assume that $|k_2^i/k_4^i|\leqslant 2, i=1,2,\cdots$ without the loss of generality. Then, from \eqref{eq: boundedness of nabla n}, we have that
\begin{equation}
\left(\frac{d^2}{4fC_6}+2 \sigma^2\right)|\Omega|\geqslant \left(\frac{d^2}{4fC_6}-\frac{k_2^i \sigma^2} {k_4^i}\right)|\Omega|\geqslant \Vert 
\nabla \n^*_{i} \Vert^2_{L^2(\Omega)}.
\end{equation}
Passing to a subsequence still labeled with $i$, there exists $\n_\infty^* \in W_\n$ such that
\begin{equation}
\n^*_i \rightharpoonup \n_\infty^* \ in \ H^1(\Omega,\mathbb{S}^2),~ \n^*_i \rightarrow \n_\infty^* \ in \ L^4(\Omega,\mathbb{S}^2). 
\end{equation}
From \eqref{vanishment of the twist term} in Lemma \ref{Lemma 4.2}, one obtains that
$$\nabla \times \n^*_\infty+\sigma \n^*_\infty=0 \ in \ \Omega,$$  where $\n^*_\infty$ is a helical configuration by Lemma \ref{configuration satisfies pde for helical configuration}.

Further,
\begin{equation}\label{eq: convergence of oseen frank elastic limit}
\begin{aligned}
&F_\infty(\n^*_\infty,\dr^*_\infty)=F_i(\n^*_\infty,\dr^*_\infty)\leqslant \liminf_{i\rightarrow \infty} F_i(\n^*_i,\dr^*_i)
\\
&\leqslant \limsup_{i\rightarrow \infty} F_i(\n^*_i,\dr^*_i)
\leqslant \limsup_{i\rightarrow \infty} F_i(\n^*_\infty,\dr^*_\infty) = F_\infty (\n^*_\infty,\dr^*_\infty),
\end{aligned}
\end{equation}
where \(F_i(\n,\dr)\) is the total energy $F$ with Frank constants \((k_j^i,j=1,2,3,4)\), and $F_\infty$ is defined in \eqref{eq: F_infty}. 
The two equalities in \eqref{eq: convergence of oseen frank elastic limit} hold since $n_\infty^*$ is the helical configuration, for which the Oseen--Frank energy vanishes. The first inequality in \eqref{eq: convergence of oseen frank elastic limit} holds with the w.l.s.c. feature, and the last inequality follows from the global minimality of $(n_i^*, \dr_i^*)$ . 
\end{proof}


Following discussions similar to \cite[Lemma 3.2]{pan2003landau}, we can prove stronger convergence results to $\n_\sigma$ as the elastic constants increase.

\begin{proposition}
Let \(\{(\mathbf{n}^*_i, \dr^*_i)\}\) be a sequence of minimizers for \(F_i(\n,\dr)\) with Frank constants \((k_j^i,j=1,2,3,4)\) satisfying the assumption in Proposition \ref{helical configuration limit}. There exists a pair $(\n_\infty^*, \dr_\infty^*)$ and a subsequence $(\n_{i_k}^*, \dr_{i_k}^*)$ such that
\begin{equation} ~\n_{i_k}^*\rightarrow \n_\infty^* \ in \ L^p(\Omega, \mathbb{S}^2),~ p\in [1,\infty), \dr_{i_k}^* \rightharpoonup \dr_\infty^* \ in \ W^{1,2}(\Omega).
\end{equation}
\end{proposition}

\begin{remark}
Proposition \ref{helical configuration limit} provides an intuitive perspective on the competition between the nematic energy, which favors a zero out-of-plane angle (since \(\n_{\sigma}\) has zero \(z\)-component), and the smectic C* energy, which favors a non-zero out-of-plane angle \(\theta_0\). If the re-scaled elastic constants are sufficiently large or if the nematic energy dominates, then the stable configuration is expected to exhibit a very small out-of-plane angle.
\end{remark}

\section{Thermotropic phase transition}\label{sec: phase transition}
In this section, we prove that the limiting Oseen-Frank free energy \eqref{eq:energy-3d vector} exhibits the symmetry-breaking transitions from the cholesteric phase (i.e., $\n= \n_\sigma$ twists in the $xy$-plane, $\dr \equiv 0$) to the helical smectic (i.e., $\n$ twists in the $xy$-plane, $\dr \not\equiv 0$) and to the smectic C$^*$ phase (i.e., $\n$ twists on the cone surface with a tilt angle, $\dr \not\equiv 0$) as temperature decreases. We do not make any assumptions on the elastic constants in \eqref{eq:energy-3d vector}, except for those stated in \eqref{assumption oseen frank}.

To this end, we impose periodic boundary conditions for \(\dr\) and natural boundary conditions for $\n$ at the top and bottom boundaries of $\Omega = \Omega_{2D}\times [0, h]$, and assume that the configuration is invariant in the \(xy\)-plane. Consequently, the domain \(\Omega = \Omega_{2D} \times [0, h]\) reduces to a one-dimensional domain \(\Omega = [0, h]\), and the energy functional $F(\n,\dr)$ \eqref{eq:energy-3d vector} reduces to
\begin{equation}\label{eq:energy-1d vector}
\begin{aligned}
&F(\n,\dr)= \int_0^h \bigg\{k_1(\nabla \cdot \n)^2+k_2|\n \cdot (\nabla \times \n)+\sigma|^2 +k_3|\n \times (\nabla \times \n)|^2\\
&+\frac{d}{2}\delta\rho^2-\frac{e}{3}\delta\rho^3+\frac{f}{4}\delta\rho^4+ \lambda_1\left(\dr_{zz} + q^2\delta\rho\right)^2
 + \lambda_2 \left(n_3^2\dr_{zz} + q^2\delta\rho\cos^2\theta_0\right)^2 \bigg\}\textrm{d} z,
\end{aligned}
\end{equation}
where $d=\alpha_2(T-T_2^*)$, $\dr_{zz}$ defines the second partial derivatives of $\dr$ with respect to $z$-direction, $n_3$ is the third component of the director $\n$. Substituting $\n=[\cos \phi(z) \cos \theta(z) , \sin \phi(z) \cos \theta(z), \sin \theta(z)]^\top$, $k_1=k_3$, into \eqref{eq:energy-1d vector}, we obtain 
\begin{equation}\label{energy theta}
\begin{aligned}
F(\phi,\theta,\dr)=\int_0^h \bigg\{k_1 \theta_z^2+\cos^2 \theta ((k_2\cos^2 \theta + k_3 \sin^2 \theta)\phi_z^2-2\sigma k_2\phi_z)\\
+\frac{d}{2}\delta\rho^2+\frac{f}{4}\delta\rho^4+ \lambda_1\left(\dr_{zz} + q^2\delta\rho\right)^2
 + \lambda_2 \left(\sin^2 \theta \dr_{zz} + q^2\delta\rho\cos^2\theta_0\right)^2 \bigg\}\d z.
\end{aligned}
\end{equation}

It is wise to convert the constrained problem with the unit-length constraint $|\n|=1$ into a problem involving the Euler angles, from a computational standpoint. However, it induces regularity issues, especially at the poles where the azimuthal angle $\phi$ is undefined. By a simple calculation,
\begin{equation*}
\int_\Omega |\nabla \n|^2 =\int_\Omega\left\{ |\nabla \theta|^2 +\cos^2 \theta |\nabla \phi|^2 \right\} \d z, ~\text{for~}\n \in W_\n,
\end{equation*}
from which we deduce that $\theta \in W^{1,2}_\Omega$. However, it is not straightforward to deduce the regularity of $\phi$. Fortunately, the energy functional in \eqref{energy theta} is always minimized by 
\begin{equation}\label{eq: phi versus theta}
\phi_z=\frac{\sigma k_2}{k_2 \cos^2 \theta+k_3 \sin^2 \theta}.
\end{equation}
Substituting \eqref{eq: phi versus theta} into \eqref{energy theta} yields the following energy functional with respect to $\theta$ and $\dr$,
\begin{equation}\label{energy theta final}
\begin{aligned}
F(\theta,\dr)=\int_0^h & \bigg\{ k_1 \theta_z^2-\frac{k_2^2\sigma^2\cos^2 \theta}{k_2 \cos^2 \theta + k_3\sin^2 \theta}+\frac{d}{2}\delta\rho^2+\frac{f}{4}\delta\rho^4\\
&+ \lambda_1\left(\dr_{zz} + q^2\delta\rho\right)^2
 + \lambda_2 \left(\sin^2 \theta \dr_{zz} + q^2\delta\rho\cos^2\theta_0\right)^2 \bigg\} \d z.
\end{aligned}
\end{equation}
The corresponding Euler-Lagrange equations are
\begin{equation}
\begin{cases}
2k_1 \theta_{zz}= \frac{2k_2^2 k_3 \sigma^2 \tan \theta}{\cos^2 \theta(k_2+k_3 \tan^2 \theta)^2} + 2\lambda_2 \left(\sin^2 \theta \dr_{zz} + q^2\delta\rho\cos^2\theta_0\right)\sin (2\theta) \dr_{zz},\\
-2\lambda_1 \dr_{zzzz}=d\dr+f(\dr)^3+4\lambda_1 q^2 \dr_{zz}+2\lambda_1 q^4 \dr + 2\lambda_2 (\sin^4 \theta \cdot \dr_{zz})_{zz}  \\ \quad \quad \quad \quad \quad 
+2\lambda_2 \left(q^2 \cos^2 \theta_0 \cdot \dr_{zz} \sin^2 \theta  + q^2 \cos^2 \theta_0 (\sin^2 \theta \cdot \dr)_{zz}+q^4 \cos^4 \theta_0 \cdot \dr \right).
\end{cases}
\label{E-L one D}
\end{equation}
with the natural (Neumann) boundary conditions for $\theta$ and periodic conditions for $\delta\rho$. Note that $\theta$ represents the out-of-plane angle, i.e., the angle between the $xy$-plane and the vertical direction. The energy density term
$\lambda_2 \left(\sin^2 \theta \dr_{zz} + q^2 \delta\rho \cos^2 \theta_0\right)^2
$
favors configurations with $\sin \theta = \cos \theta_0$. Therefore, in the deep smectic C* phase, i.e., the phase with non-zero $\dr$ and non-zero out-of-plane angle, one can expect $\theta \approx \tfrac{\pi}{2} - \theta_0$.

\begin{remark}
With the one-constant approximation $k_1=k_2=k_3$, we have $\phi_z=\sigma$ and $\phi(z)=\sigma z+z_0$, then $(\theta \equiv 0, \phi=\sigma z)$ corresponds to a helical configuration with the director field $\n=[\cos \phi(z) \cos \theta(z) , \sin \phi(z) \cos \theta(z), \sin \theta(z)]^\top=(\cos \sigma z, \sin \sigma z, 0)$, referred to as $\n_\sigma$ in the previous section.
\end{remark}

The admissible spaces for $\theta,\dr$ are defined to be
\begin{equation}\label{admissible space of theta and dr}
\begin{aligned}
&W_\theta=\left\{ \theta \in \left[-\frac{\pi}{2},\frac{\pi}{2}\right], \theta \in W^{1,2}_\Omega, \theta_z(0)=\theta_z(h)=0\right\},\\
&W_\dr=\left\{ \dr \in W^{2,2}_\Omega, \dr(0)=\dr(h), \dr_z (0)=\dr_z (h), \dr_{zz}(0)=\dr_{zz}(h)\right\}.
\end{aligned}
\end{equation}
In the following, we assume that $k_1=k_2=k_3$ for brevity, but the results also apply to the general case in \eqref{relationship between Q and v parameters}.

\begin{proposition}[The loss of the stability of the cholesteric phase with low temperature]\label{cholesteric loses stability}
Let $k_1=k_2=k_3, \lambda_1, \lambda_2, f, 0<\theta_0<\pi/2$ be positive constants, and let $q=\frac{2\pi n_0}{h}$ for a fixed positive integer $n_0$, where $n_0=1,2,3,\cdots$. For high temperatures such that $d+2\lambda_2 q^4\cos^2 \theta_0>0$, where $d=\alpha_2(T-T^*_2)$, the cholesteric helical configuration $(\theta\equiv 0, \dr \equiv 0)$ (corresponds to  $(\n=\n_{\sigma},\dr \equiv 0)$) is a stable critical point of \eqref{energy theta final}. The cholesteric helical configuration loses its stability as the temperature decreases such that $d+2\lambda_2 q^4\cos^2 \theta_0<0$.
\end{proposition}
\begin{proof}
One can check that $(\theta\equiv 0, \dr \equiv 0)$ (i.e., the cholesteric nematic phase $(\n=\n_{\sigma},\dr \equiv 0)$) is always a critical point of \eqref{energy theta final} for any temperature $T$. To investigate the stability of $(\theta\equiv 0, \dr \equiv 0)$ near $T=T^*_2$, we calculate the second variation of \eqref{energy theta final} at $(\theta \equiv 0,\dr \equiv 0)$ for a perturbation, $(\eta_1,\eta_2)$, where $\eta_2$ satisfies the periodic boundary condition and
\begin{equation}
\delta^2 F(\eta_1,\eta_2)=\int_0^h \left\{2k_1 \eta_{1z}^2 + 2k_1\sigma^2 \eta_1^2+(d(T)+2\lambda_2 q^4\cos^4 \theta_0)\eta_2^2+2\lambda_1\left(\eta_{2zz}+q^2\eta_2\right)^2\right\}.
\label{second variation}
\end{equation}
The stability of the cholesteric helical configuration is measured by the minimum eigenvalue of $\delta^2 F$,
\begin{equation}
\mu_T=\inf_{\eta_1\in W^{1,2}_\Omega,\eta_2 \in W_\dr}\frac{\delta^2F(\eta_1,\eta_2)}{\int_0^h \eta_1^2 + \eta_2^2 \mathrm{d}x}.
\label{eigenvalue}
\end{equation}
If $\mu_T<0$, the cholesteric phase is unstable while the cholesteric phase is stable if $\mu_T>0$.

Any perturbation with a non-zero $\eta_1$ is always a stable direction. Thus, we only consider the perturbation $(0,\eta_2)$. The Fourier expansion of the function, $\eta_2$ in $\Omega=[0,h]$ is given by,
\begin{equation}
\eta_2=w_0/2+\sum_{n=1}^\infty w_n \cos\left(\frac{2\pi n z}{h}\right) + v_n \sin\left(\frac{2\pi n z}{h}\right).
\label{fourier expansion}
\end{equation}
By substituting \eqref{fourier expansion} into \eqref{second variation}, we obtain
\begin{equation}
\begin{aligned}
\delta^2 E_{1D}(0,\eta_2)=&\frac{h}{2} \left(\frac{d+2\lambda_2 q^4\cos^4 \theta_0+2\lambda_1q^4}{2}\right)w_0^2\\
&+\frac{h}{2}\sum_{n=1}^\infty \left[2\lambda_1\left(\frac{4\pi ^2 n^2}{h^2}-q ^2\right)^2+d+2\lambda_2 q^4\cos^4 \theta_0\right](w_n^2+v_n^2).
\end{aligned}
\end{equation}
Therefore, $(0,\eta)$ is an eigenfunction of \eqref{second variation} if and only if 
\begin{equation}
    2(d+2\lambda_2 q^4\cos^4 \theta_0) \eta + 4\lambda_1 \eta_{xxxx}+8\lambda_1q^2\eta_{xx}+4\lambda_1q^4\eta=\lambda\eta.
    \label{KKT}
\end{equation}

One can verify that \eqref{KKT} is the first order optimal condition (or the so-called KKT condition in \cite{boyd2004convex}) of \eqref{eigenvalue}.
By substituting \eqref{fourier expansion} into \eqref{KKT}, we get that $\eta\equiv 1$, $\eta= \cos(\frac{2\pi n z}{h})$ and $\eta= \sin(\frac{2\pi n z}{h})$, $n=1,2,3\cdots$ are the eigenvectors of $\delta^2 F$ with eigenvalues $\mu=d+2\lambda_2 q^4\cos^4 \theta_0+2\lambda_1q^4$ and $d+2\lambda_2 q^4\cos^4 \theta_0+2\lambda_1\left(\frac{4\pi ^2 n^2}{h^2}-q^2\right)^2$, $n=1,2,3\cdots$, respectively.
For $n_0\in\mathbb{Z^+}$ such that $\left(\frac{4\pi ^2 n_0^2}{h^2}-q^2\right)^2 = 0$,
both $\eta=\sin(\frac{2\pi n_0 z}{h}) = \sin\left( q z \right)$ and $\eta= \cos(\frac{2\pi n_0 z}{h}) = \cos\left(q z \right)$ are eigenvectors corresponding to the minimum degenerate eigenvalue $\mu=d+2\lambda_2 q^4 \cos^4\theta_0$.
For $T\geqslant T_2^*$, i.e., $d+2\lambda_2 q^4\cos^4 \theta_0 \geqslant0$, the second variation is always positive, i.e., the cholesteric phase is stable.
For $T< T_2^*$ (recall that $d=\alpha_2(T-T^*_2)$), i.e., $d+2\lambda_2 q^4\cos^4 \theta_0 <0$, the eigenvector $\eta \equiv 1$ is an unstable eigendirection if and only if the corresponding eigenvalue $d+2\lambda_2 q^4\cos^4 \theta_0+2\lambda_1 q^4<0$ is negative, and the eigenvectors $\sin(\frac{2\pi n z}{h})$ and $\cos(\frac{2\pi n z}{h})$, $n=1,2,3\cdots$, are unstable eigen-directions if and only if the corresponding eigenvalue $d+2\lambda_2 q^4\cos^4 \theta_0+2\lambda_1 \left(\frac{4\pi ^2 n^2}{h^2}-q^2\right)^2$ is negative. Thus, the Morse index of the cholesteric phase, i.e., the number of eigenvectors corresponding to negative eigenvalues is given by
\begin{equation}\label{eq: Morse index of nematic}
i_{cholesteric} = 
    2\times card(\mathbb{N}_{cholesteric})+m_0,
\end{equation}
where
\begin{equation}\label{eq: index constraint}
\mathbb{N}_{cholesteric} = \left\{n\in\mathbb{Z^+}:d+2\lambda_2 q^4\cos^4 \theta_0+2\lambda_1\left(\frac{4\pi ^2 n^2}{h^2}-q^2\right)^2<0\right\}
\end{equation}
and $card(\mathbb{N}_{cholesteric})$ is the cardinal number of $\mathbb{N}_{cholesteric}$. If $d+2\lambda_2 q^4 \cos^4 \theta_0+2\lambda_1 q^4\geqslant 0$, $m_0=0$; otherwise $m_0=1$, i.e., $\eta \equiv 1$ is an unstable eigen-direction.
As the parameter $d$ (or the temperature)
decreases, more positive integers satisfy the constraint in \eqref{eq: index constraint}, and the Morse index of the cholesteric phase, $i_{cholesteric}$, increases. 
\end{proof}

The aforementioned calculations show that the cholesteric phase loses its stability as the temperature decreases. We  now demonstrate that when the cholesteric phase loses stability, it bifurcates into a more stable smectic phase with non-zero $\dr$. In the proof of Proposition \ref{cholesteric loses stability}, we note that the minimum eigenvalue of the cholesteric phase is degenerate, which can pose technical difficulties in applying methods from bifurcation theory \cite{chow2012methods}. To circumvent this issue, we construct the following function space
\begin{equation}
    V=W_{\dr} \cap W^{1,2}_{0,\Omega},
\end{equation}
which restricts $\eta=\cos(q z)$ from serving as an eigenvector and then simplifies the minimum eigenvalue at the cholesteric phase.

\begin{proposition}\label{bifurcation proof}
    Let $k_1=k_2=k_3, \lambda_1, \lambda_2, f, 0<\theta_0<\pi/2$ be positive constants, and let $q=\frac{2\pi n_0}{h}$ for a fixed positive integer $n_0$, where $n_0=1,2,3,\cdots$. The Euler-Lagrange equations \eqref{E-L one D} associated with the energy functional \eqref{energy theta final}, exhibit a pitchfork bifurcation at $d=-2\lambda_2 q^4 \cos^4 \theta_0$ and $(\theta \equiv 0, \dr \equiv 0)$ in $W_\theta \times V$. More precisely,
there exist positive numbers $\epsilon,\delta$ and two smooth maps
\begin{equation}\label{eq: relation t and d}
    t\in (-\delta,\delta)\rightarrow d(t)+ 2\lambda_2 q^4 \cos^4 \theta_0 \in(-\epsilon,\epsilon), t\in (-\delta,\delta)\rightarrow w_t \in V,
\end{equation}
where $d(t)$ is a monotonically decreasing function of $t$, such that all the pairs $(d,\theta,\dr)\in R\times V$ satisfying 
$$
(\theta,\dr) \text{ } is \text{ } a \text{ } solution \text{ } to \text{ } \eqref{E-L one D}, |d+2\lambda_2q^4\cos^4\theta_0|<\epsilon,  \Vert w_t \Vert_{W^{2,2}_\Omega}\leqslant \epsilon
$$
are 
\begin{equation}\label{analytical configuration near phase transition temperature}
\begin{aligned}
&either\text{ } cholesteric \text{ } phase: (\theta \equiv 0, \dr \equiv0) \text{ }\\
or\text{ } smectic & \text{ } phases: (\theta \equiv 0,\dr=\pm \left(t \sin(q z)+ t^2 w_t\right)).
\end{aligned}
\end{equation}
\end{proposition}

\begin{proof}
The proof is similar to \cite[Proposition 3.3]{shi2025modified}. However, we have a coupled system with an additional elastic energy $f_{angle}$; we provide a complete proof here. To show that a pitchfork bifurcation arises at $d=-2\lambda_2 q^4 \cos^4 \theta_0$, we apply the Crandall-Rabinowitz bifurcation theorem \cite{crandall1971bifurcation} to the operator $\mathcal{F}:\mathbb{R}\times W_\theta \times V \rightarrow W^{-1,2}_\Omega \times W^{-2,2}_\Omega$ ($W^{-i,2}_\Omega, i=1,2$ is the dual space of $W^{i,2}_\Omega$) defined as
\begin{equation*}
    \mathcal{F}(d,v,w):=(\mathcal{F}_1(v,w),\mathcal{F}_2(d,v,w)),
\end{equation*}
with
\begin{equation}\label{operator}
\begin{aligned}
\mathcal{F}_1(v,w)=-2k_1 v_{zz}+k_1 \sigma^2 \sin(2 v) + 2 \lambda_2 \left(w_{zz} \sin^2 v  + q^2 w \cos^2\theta_0\right)\sin (2v) w_{zz}\\
\mathcal{F}_2(d,v,w)=2\lambda_1 w_{zzzz}+dw+fw^3+4\lambda_1 q^2 w_{zz}+2\lambda_1 q^4 w + 2\lambda_2 (w_{zz}\sin^4 v)_{zz} \\
+ 2\lambda_2 q^2 \cos^2 \theta_0\cdot w_{zz} \sin^2 v + 2\lambda_2 q^2 \cos^2 \theta_0 (\sin^2 v \cdot w)_{zz}+2\lambda_2 q^4 \cos^4 \theta_0 w,
\end{aligned}
\end{equation}
which arises from perturbing the cholesteric state
$(\theta_c\equiv 0, \delta\rho_c\equiv 0)$ by $(v,w)$
in \eqref{E-L one D}, i.e.,
substituting
$(\theta,\delta\rho)
=(\theta_c\equiv 0, \delta\rho_c\equiv 0)+(v,w)=(v,w)$
into \eqref{E-L one D}.

We now check the four assumptions of \cite[Theorem 1.7]{crandall1971bifurcation}.\\
(a) $\mathcal{F}(d,0,0)=0$;\\
(b) The partial derivatives $D_d\mathcal{F}, D_{(v,w)}\mathcal{F}, D_{d,(v,w)}\mathcal{F}$ exist and are continuous;\\
(c) $dim\left(\frac{W^{-1,2}(\Omega)\times W^{-2,2}(\Omega)}{Range(D_{(v,w)}\mathcal{F}(0,0))}\right)=\text{dim}\left(\text{Kernel}\left(D_{(v,w)}\mathcal{F}(0,0)\right)\right)=1$;\\
(d) $ D_{d,(v,w)}\mathcal{F}w_0\notin Range(D_{{v,w}}\mathcal{F}(0,0))$, where $w_0\in \text{Kernel}\left(D_w\mathcal{F}(0,0)\right)$. 

$\mathcal{F}(d,0,0)=0$ holds for all $d\in \mathbb{R}$. We have
\begin{equation}\label{differential operators}
\begin{cases}
    D_d\mathcal{F}(d,v,w)=(0,w),\\
    D_{(v,w)}\mathcal{F}(d,v,w)=(D_{11},D_{12};D_{21},D_{22})\\
    D_{d,(v,w)}\mathcal{F}(d,v,w)=(0,0;0,1),
\end{cases}
\end{equation}
where $D_{ij}=: W^{j,2}(\Omega)\rightarrow W^{-i,2}(\Omega)$ ($D_{11}=D_v \mathcal{F}_1$,$D_{12}=D_w \mathcal{F}_1$,$D_{21}=D_v \mathcal{F}_2$,$D_{22}=D_w \mathcal{F}_2$) are calculated as  
\begin{equation*}
\begin{aligned}
D_{11}[u]=&-2k_1D_{zz}u+2k_1\sigma^2\cos(2v)u+f_1(\cos v,\sin v)w_{zz}u\\
&+f_2(\cos\theta,\sin\theta)(w_{zz})^2u,\\
D_{22}[u]=&2\lambda_1 D_{zzzz}u+du+3fw^2u+4\lambda_1 q^2 u_{zz}\\
&+2\lambda_1 q^4 u + 2\lambda_2 (\sin^4 v\cdot D_{zz}u)_{zz},\\
&+ 2\lambda_2 q^2 \cos^2 \theta_0\cdot u_{zz} \sin^2 v + 2\lambda_2 q^2 \cos^2 \theta_0 (\sin^2 v \cdot u)_{zz}+2\lambda_2 q^4 \cos^4 \theta_0 u,
\\
D_{12}[u]=&4\lambda_2\sin^2 v \cdot u_{zz}\sin(2v)w_{zz}+2 \lambda_2  q^2 w \cos^2\theta_0\sin (2v) u_{zz}, D_{21}=D_{12}^*.
\end{aligned}
\end{equation*}
where $f_i(\cos v,\sin v) \in L^\infty(\Omega)$ is the combination of trigonometric functions in the polynomial form. The operators in \eqref{differential operators} are continuous.
Actually, for any $u_1 \in W^{1,2}(\Omega) \hookrightarrow L^\infty(\Omega)$, $w \in W^{2,2}(\Omega)$, $w_{zz} \in L^2(\Omega)$, $w_{zz}^2 \in L^1(\Omega)$, we have
\begin{equation*}
\begin{aligned}
&\Vert D_{11}(d,v,w)[u_1] \Vert_{W^{-1,2}(\Omega)}=\sup_{\Vert u_2 \Vert_{W^{1,2}(\Omega)}\leqslant 1} \left| \int_\Omega D_{11}(d,v,w)[u_1]\cdot u_2 \d z \right|\\
&\leqslant
 \sup_{\Vert u_2 \Vert_{W^{1,2}(\Omega)}\leqslant 1}2k_1\left|\int_\Omega D_z u_1 D_z u_2 \d z \right|+2k_1 \sigma^2\Vert u_1 \Vert_{W^{1,2}(\Omega)}\\
 &\quad+C\sup_{\Vert u_2 \Vert_{W^{1,2}(\Omega)}\leqslant 1}\left|\int_\Omega w_{zz}u_1 u_2 \d z\right|+C\sup_{\Vert u_2 \Vert_{W^{1,2}(\Omega)}\leqslant 1}\left|\int_\Omega w_{zz}^2 u_1 u_2 \d z\right|\\
&\leqslant(2k_1+2k_1\sigma^2)\Vert u_1 \Vert_{W^{1,2}(\Omega)}+C \Vert w_{zz} \Vert_{L^2(\Omega)} \Vert u_1 \Vert_{W^{1,2}(\Omega)} \\
&\quad+C(\Omega) \Vert w_{zz}^2 \Vert_{L^1(\Omega)} \Vert u_1 \Vert_{W^{1,2}(\Omega)},\\
&\leqslant \left(2k_1+2k_1\sigma^2+C(\Omega)\Vert w_{zz} \Vert_{L^2(\Omega)}+C\Vert w_{zz}^2 \Vert_{L^1(\Omega)}\right)\Vert u_1 \Vert_{W^{1,2}(\Omega)}
\end{aligned}
\end{equation*}
which indicates that $D_{11}$ is continuous. For any $u_1 \in W^{2,2}(\Omega) \hookrightarrow L^\infty(\Omega)$, $w \in W^{2,2}(\Omega)$, we have
\begin{equation}
\begin{aligned}
&\Vert D_{22}(d,v,w)[u_1] \Vert_{W^{-2,2}(\Omega)}=\sup_{\Vert u_2 \Vert_{W^{2,2}(\Omega)}\leqslant 1} \left| \int_\Omega D_{22}(d,v,w)[u_1]\cdot u_2 \d z \right|\\
&\leqslant
 \sup_{\Vert u_2 \Vert_{W^{2,2}(\Omega)}\leqslant 1}2\lambda_1\left|\int_\Omega D_{zz} u_1 D_{zz} u_2 \d z \right|\\
 &\quad+ \sup_{\Vert u_2 \Vert_{W^{2,2}(\Omega)}\leqslant 1}2\lambda_2\left|\int_\Omega \sin^4 v D_{zz} u_1 D_{zz} u_2 \d z \right|\\
 &\quad+\left(d+3f\Vert w^2\Vert_{L^\infty(\Omega)}+4\lambda_1 q^2+2\lambda_1 q^4+2\lambda_2 q^4\right)\Vert u_1 \Vert_{L^2(\Omega)}\\
 & \quad+ \sup_{\Vert u_2 \Vert_{W^{2,2}(\Omega)}\leqslant 1} 2\lambda_2 q^2 \|D_{zz}u_1\|_{L^2(\Omega)}\|u_2\|_{L^2(\Omega)} + 2\lambda_2 q^2 \|u_1\|_{L^2(\Omega)} \|D_{zz}u_2\|_{L^2(\Omega)}  \\
 &\leqslant \left((2+4q^2+2q^4)(\lambda_1+\lambda_2)+d+3f\Vert w^2\Vert_{L^\infty(\Omega)}\right)\Vert u_1 \Vert_{W^{2,2}(\Omega)}
\end{aligned}
\end{equation}
which indicates that $D_{22}$ is continuous. Similarly, $D_{12}$ and $D_{21}$ are also continuous.

In order to check that $\mathcal{F}$ satisfies assumption (c), one can calculate the kernel space of 
\begin{equation}\label{eq: dvw000}
\begin{aligned}
&D_{(v,w)}\mathcal{F}(-2\lambda_2 q^4 \cos^4 \theta_0 ,0,0)\\&=(-2k_1D_{zz}+2k_1\sigma^2,0;0,-2\lambda_2 q^4 \cos^4 \theta_0+2\lambda_1(D_{zz}+q^2)^2+2\lambda_2 q^4 \cos^4 \theta_0 )
\\
&=(-2k_1D_{zz}+2k_1\sigma^2,0;0,\lambda_1(D_{zz}+q^2)^2)
\end{aligned}
\end{equation}
in $W_\theta \times V$, the spectrum of $-2k_1D_{zz}+2k_1\sigma^2$ belongs to $(0,\infty)$, and
\begin{equation}
\text{dim}\left(\text{Kernel}\left(2\lambda_1(D_{zz}+q^2)^2\right)\right)=\text{dim}\left(\{w=k \sin(q z),~k \in \mathbb{R}\}\right)=1.
\end{equation}
Hence, 
\begin{equation}
\textrm{dim}\left(\textrm{Kernel}\left(D_{(v,w)}\mathcal{F}(-2\lambda_2 q^4 \cos^4 \theta_0,0,0)\right)\right)=\text{dim}\left(\text{Kernel}\left(2\lambda_1(D_{zz}+q^2)^2\right)\right)=1.
\end{equation}
Since we have periodic boundary condition for $\dr$ and Neumann boundary condition for $\theta$, it is direct to see that $D_{(v,w)}\mathcal{F}(-2\lambda_2 q^4 \cos^4 \theta_0,0,0)$ defined in \eqref{eq: dvw000} is a self-adjoint operator, and hence it is a Fredholm operator of index $0$ \cite{ize1976bifurcation}. We have 
\begin{equation}
   \begin{aligned} &\text{dim}\left(\frac{W^{-1,2}(\Omega)\times W^{-2,2}(\Omega)}{\text{Range}(D_{(v,w)}\mathcal{F}(-2\lambda_2 q^4 \cos^4 \theta_0,0,0))}\right)\\
   &=\text{dim}\left(\text{Kernel}\left(D_{(v,w)}\mathcal{F}(-2\lambda_2 q^4 \cos^4 \theta_0,0,0)\right)\right)=1,
   \end{aligned}
\end{equation}
which satisfies the assumption (c).
We now proceed to checking the last assumption (d)
\begin{equation*} 
    D_{d,(v,w)}\mathcal{F}[(0,\sin(q z)]=(0,\sin(q z))\notin \text{range} D_{(v,w)}\mathcal{F}(0,0,0),
\end{equation*}
i.e., the following differential equation
\begin{equation}
2\lambda_1(D_{zz}+q^2)^2 w = \sin(q z)
\label{assumption d}
\end{equation}
does not have a solution that satisfies the periodic boundary condition. One can check that the general solution of \eqref{assumption d} is
\begin{equation}
    w=-\frac{x^2\sin(q z)}{16\lambda_1 q^2}+k \sin(q z),~k \in \mathbb{R},
\end{equation}
and it cannot satisfy the periodic boundary conditions with any $k\in \mathbb{R}$. All the assumptions of the Crandall-Rabinowitz theorem are now fulfilled, and thus the proposition follows directly from \cite{crandall1971bifurcation}.
\end{proof}

\begin{figure}
\centering
\includegraphics[width=.9\textwidth]{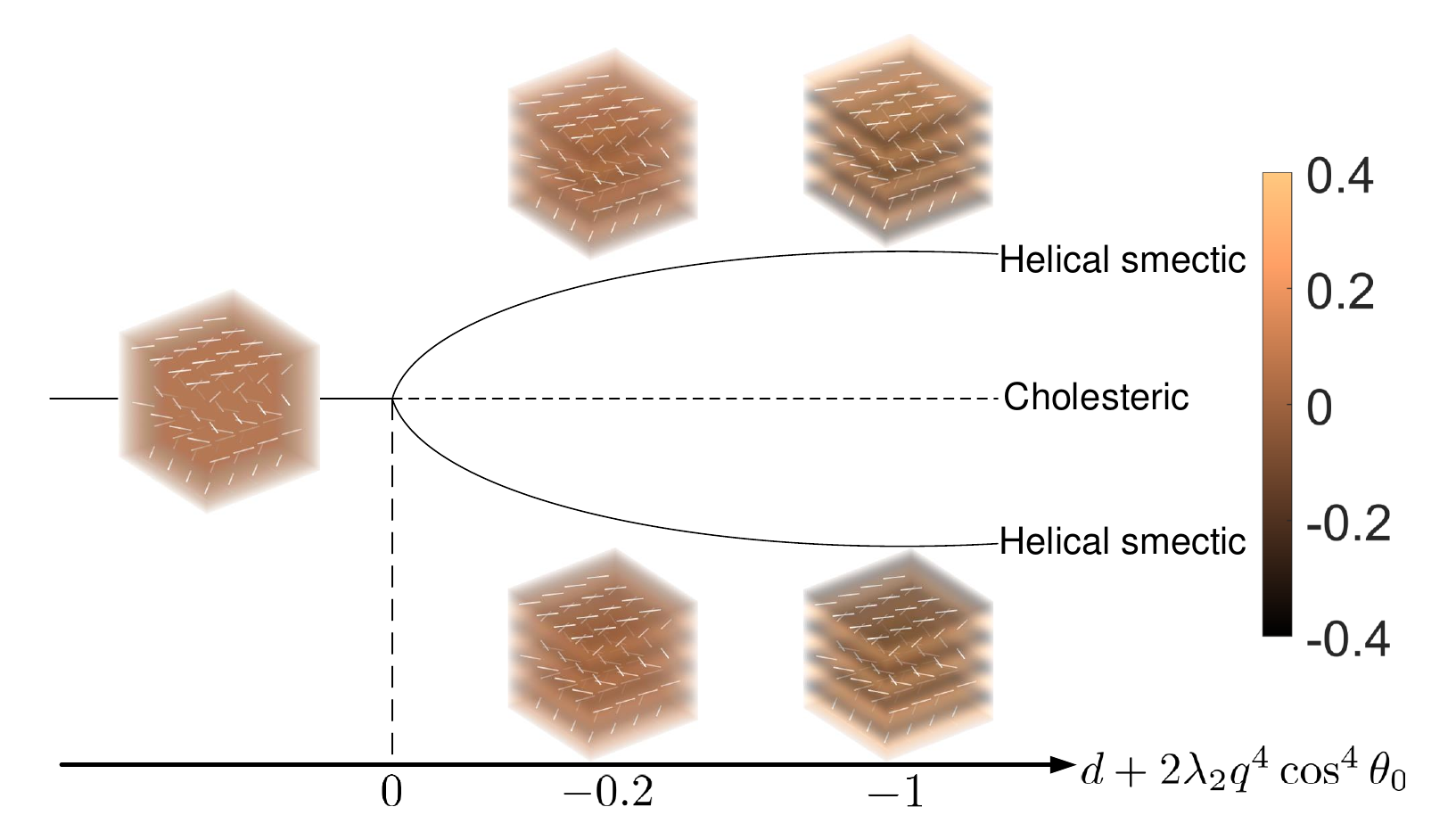}
\caption{Schematic illustration of the  
the Cholesteric-Helical Smectic phase transition with $d=T+10$, $e=0$, $f=10$, $k_1=k_2=k_3=k=0.025$, $h=2\pi$, $\sigma=q=4$, $\lambda_1=\lambda_2=0.001$, $\theta_0=\pi/9$, and the pitchfork bifurcation for $d+2\lambda_2 q^4 \cos^4\theta_0<0$. The solid black line denotes a stable phase, while the dashed black line denotes an unstable phase in all figures. We numerically calculate the minimizer ($\dr_{min}$,$\theta_{min}$) of \eqref{energy theta final} with various $d$. For better visualization, we plot the 3D $xy$-invariants: $\Tilde{\theta}(x,y,z)\equiv \theta(z)$ and $\Tilde{\dr}(x,y,z) \equiv \dr(z)$.}
\label{C-HS transition}
\end{figure}

In Figure~\ref{C-HS transition}, we numerically calculate the Cholesteric-Helical Smectic bifurcation, accomplished using the sine spectral method for $\dr$ and the cosine spectral method for $\theta$ \cite{shen2011spectral}. For $d+2\lambda_2 q^4 \cos^4\theta_0>0$, the minimum eigenvalue at the cholesteric phase, calculated both numerically and analytically, is both simple and positive, indicating the stability. When $d+2\lambda_2 q^4 \cos^4 \theta_0=0$, a simple zero eigenvalue emerges with the eigenvector $\eta=\sin(q z)$.
As $d+2\lambda_2 q^4 \cos^4 \theta_0$ becomes negative, the cholesteric phase loses its stability and bifurcates into two helical smectic phases, corresponding to $\dr=t\sin(q z)+t^2 w_t$ and $\dr=-t\sin(q z)-t^2 w_t$ respectively, in the pitchfork bifurcation. With Proposition \ref{bifurcation proof}, we can directly have the following symmetry-breaking transition theorem between cholesteric and helical smectic phases.

\begin{theorem}[Cholesteric-Helical Smectic phase transition]
Let $k_1=k_2=k_3, \lambda_1, \lambda_2, f, 0<\theta_0<\pi/2$ be positive constants, and let $q=\frac{2\pi n_0}{h}$ for a fixed positive integer $n_0$, where $n_0=1,2,3,\cdots$. As the temperature decreases, the energy functional \eqref{energy theta final} exhibits a cholesteric-smectic phase transition at $d=-2\lambda_2 q^4 \cos^4 \theta_0$, i.e., the cholesteric phase (with zero $\dr$) is stable for $d\geqslant-2\lambda_2 q^4 \cos^4 \theta_0$, but loses stability when $d<-2\lambda_2 q^4 \cos^4 \theta_0$ and the helical smectic phase with a non-zero $\dr$ is stable.
\end{theorem}

In the above discussion of cholesteric nematic to helical smectic phase transitions, it is observed that near the critical temperature, the helical smectic phase exhibits a zero out-of-plane director component (see the
analytical solution in \eqref{analytical configuration near phase transition temperature} and the numerical result in Figure \ref{C-HS transition}, where $\theta \equiv 0$ near the bifurcation point). In the following, we show that as the layer structure becomes increasingly pronounced, the director gradually twists along a conical surface, i.e., a Helical Smectic-Smectic C* phase transition occurs as the temperature further decreases. 

\begin{proposition}[Existence of the helical smectic phase]
Let $k_1=k_2=k_3$, $\lambda_1, \lambda_2, f, 0<\theta_0<\pi/2$ be positive constants, and let $q=\frac{2\pi n_0}{h}$ for a fixed positive integer $n_0$, where $n_0=1,2,3,\cdots$. The helical smectic phase ($\dr \not\equiv 0,\theta \equiv 0$) exists as a weak solution of \eqref{E-L one D} when $d<d_0=-2\lambda_2 q^4 \cos^4 \theta_0$ (or $T<T_2^*+\frac{d_0}{\alpha_2}$).
\end{proposition}
\begin{proof}
We can directly show that ($\theta \equiv 0$) is always a solution of the Euler--Lagrange equation in \eqref{E-L one D}. With a similar argument as in Proposition \ref{existence proof}, one can prove that the following functional
\begin{equation}
F_2(\dr)=\int_0^h  \left\{\frac{d}{2}(\dr)^2+\frac{f}{4}(\dr)^4\\
+ \lambda_1\left(\dr_{zz} + q^2\delta\rho\right)^2
 + \lambda_2 \left(q^2\delta\rho\cos^2\theta_0\right)^2\right\},
\end{equation}
admits a global minimizer $\dr^*$. Then, $(\dr^*,\theta \equiv 0)$ is a weak solution of \eqref{E-L one D}.

If $d+2\lambda_2 q^4 \cos^4 \theta_0<0$, we have $\dr^*\not\equiv 0$ since
\begin{equation*}
F_2(\dr^*)\leqslant F_2(\epsilon \sin(q z))=\frac{(d+2\lambda_2 q^4 \cos^4 \theta_0)\epsilon^2h}{4}+\frac{3f\epsilon^4h}{32}<F_2(\dr\equiv 0)=0
\end{equation*}
for $0<\epsilon<\sqrt{-\frac{8(d+2\lambda_2 q^4 \cos^4 \theta_0)}{3f}}$.
\end{proof}

From \eqref{eq: relation t and d}, $t$ increases as $d$ (or the temperature) decreases. This indicates that the layer structure becomes more pronounced at lower temperatures. Therefore, analyzing the system under decreasing temperature is qualitatively equivalent to studying phase transitions as $t$ increases.
By \eqref{analytical configuration near phase transition temperature} in Proposition \ref{bifurcation proof}, the leading term of the layer structure is $t\sin q z$. In the next proposition, with the assumption $\dr=t\sin qz$, we show that the helical smectic phase $(\dr\not\equiv 0,\theta \equiv 0)$ loses stability with respect to the smectic C* phase $(\dr\not\equiv 0,\theta \not\equiv 0)$, as the temperature further decreases.

\begin{proposition}[Helical Smectic-Smectic C* phase transition]\label{twist plane changes}
Let $k_1=k_2=k_3, \lambda_1, \lambda_2, f, 0<\theta_0<\pi/2$ be positive constants, and let $q=\frac{2\pi n_0}{h}$ for a fixed positive integer $n_0$, where $n_0=1,2,3,\cdots$. Assume that $\dr=t \sin qz, t\geqslant 0$ (i.e., neglect second-order terms in \eqref{analytical configuration near phase transition temperature}), then $\theta \equiv 0$ is a critical point of the energy \eqref{energy theta final} for any $t\geqslant 0$. As $t$ increases (or temperature decreases), $\theta \equiv 0$ (or the helical smectic phase) is stable for $t< t_1=\sqrt{\frac{k_1 \sigma^2}{2\lambda_2 q^4\cos^2 \theta_0}}$ but unstable for $t> t_2=\sqrt{\frac{2q h k_1 \sigma^2}{(2q h-\sin(2q h))\lambda_2 q^4\cos^2 \theta_0}}$. Assuming a constant value of $\theta$ (or constant out-of-plane director component), the optimal twist plane is 
\begin{equation}
\theta^* \equiv \arcsin{\sqrt{\max\left(\cos^2\theta_0-\frac{k_1h\sigma^2}{\lambda_2t^2q^4\left(h-\frac{\sin(2q h)}{2q}\right)},0\right)}}.
\end{equation}
\end{proposition}

\begin{proof}
By substituting $\dr=t\sin q z$ into \eqref{energy theta final}, we have the following free energy functional
\begin{equation*}
F(\theta)=\int_0^h\left\{ k_1 \theta_z^2-k_1 \sigma^2 \cos^2 \theta+ \lambda_2 t^2 q^4 \sin ^2 (q z) \left(\cos^2\theta_0-\sin^2 \theta \right)^2\right\} + C.
\end{equation*}
and the Euler-Lagrange equation is 
\begin{equation*}
k_1 \theta_{zz}= k_1 \sigma^2 \sin (2\theta)/2-\lambda_2 t^2 q^4 \sin^2 (q z)(\cos^2 \theta_0-\sin^2 \theta)\sin (2\theta).
\end{equation*}
The helical configuration $\theta \equiv 0$ satisfies the Euler-Lagrange equation for any $t\geqslant 0$. The second variation at $\theta \equiv 0$ is 
\begin{equation}
\begin{aligned}
\delta ^2 F(\bar{\theta})&=2\int_0 ^h \left\{k_1 \bar{\theta}_z^2+k_1 \sigma^2 \bar{\theta}^2-2\lambda_2 t^2 q^4 \sin^2(q z) \cos^2 \theta_0 \cdot \bar{\theta}^2\right\}\d z\\
&\geqslant 2\int_0 ^h \left\{k_1 \bar{\theta}_z^2+k_1 \sigma^2 \bar{\theta}^2-2\lambda_2 t^2 q^4 \bar{\theta}^2 \cos^2 \theta_0 \right\}\d z.
\end{aligned}
\end{equation}
Thus, if $t<t_1=\sqrt{\frac{k_1 \sigma^2}{2\lambda_2 q^4 \cos^2 \theta_0}}$, then $\delta ^2 F(\bar{\theta})\geqslant C\Vert \bar{\theta} \Vert_{W^{1,2}(\Omega)}^2$, which means that the helical configuration $\theta \equiv 0$ is stable.
If $t\geqslant t_2=\sqrt{\frac{2q h k_1 \sigma^2}{(2q h-\sin(2q h))\lambda_2 q^4\cos^2 \theta_0 }}$,
\begin{equation}
\inf_{\bar{\theta} \in W_\theta}\frac{\delta ^2 F(\bar{\theta})}{\int_0^h \bar{\theta}^2 \mathrm{d}z}\leqslant \frac{\delta ^2 F(\bar{\theta}\equiv 1)}{h}=2\left(k_1\sigma^2-\left(1-\frac{\sin(2 q h)}{2 q h}\right)\lambda_2 t^2 q^4 \cos^2 \theta_0\right)<0,
\label{eigenvalue_2}
\end{equation}
which means the helical configuration is unstable.

Especially, if $\theta$ is a constant function, then 
\begin{equation}
\begin{aligned}
F(\theta)&=-k_1 h \sigma^2 \cos^2 \theta+ \lambda_2 t^2 q^4 \left(\frac{h}{2}-\frac{\sin(2q h)}{4q}\right)\left(\cos^2 \theta_0-\sin^2 \theta \right)^2 + C,
\end{aligned}
\end{equation}
and 
\begin{equation}\label{optimal twist plane}
\theta^*=\argmin_\theta F(\theta)=\arcsin{\sqrt{\max\left(\cos^2\theta_0-\frac{k_1h\sigma^2}{\lambda_2 t^2 q^4\left(h-\frac{\sin(2q h)}{2q}\right)},0\right)}}.
\end{equation}
\end{proof}

\begin{remark}
The tilt angle results from the competition between nematic and smectic energies. The nematic energy favors a helical configuration that twists in-plane with $\theta=0$, whereas the smectic energy favors a configuration that twists on a conical surface. In \ref{helical configuration limit}, we prove that the director tends toward a helical configuration defined by Lemma \ref{configuration satisfies pde for helical configuration}, consistent with the optimal twist plane in \eqref{optimal twist plane}, as $k_1 = k_2=k_3=k \to \infty$ ($\theta^* \equiv 0$ as $k_1\to\infty$ in \eqref{optimal twist plane}). Proposition \ref{twist plane changes} also indicates that although the layer structure encourages the out-of-plane director, the director does not immediately point out of plane upon the appearance of the smectic structure (see Figure \ref{fig: optimal twist plane}). Actually, \(\theta \equiv 0\) remains stable even when the layer structure is present but not pronounced $(t<t_1)$. Thus, this model captures the phase transition from Cholesteric (in-plane twist, no layer structure) to Helical Smectic (in-plane twist, with layer structure) and finally to Smectic C* (conical surface twist, with layer structure), as the temperature decreases.
\end{remark}
\begin{figure}
\centering
\includegraphics[width=.85\textwidth]{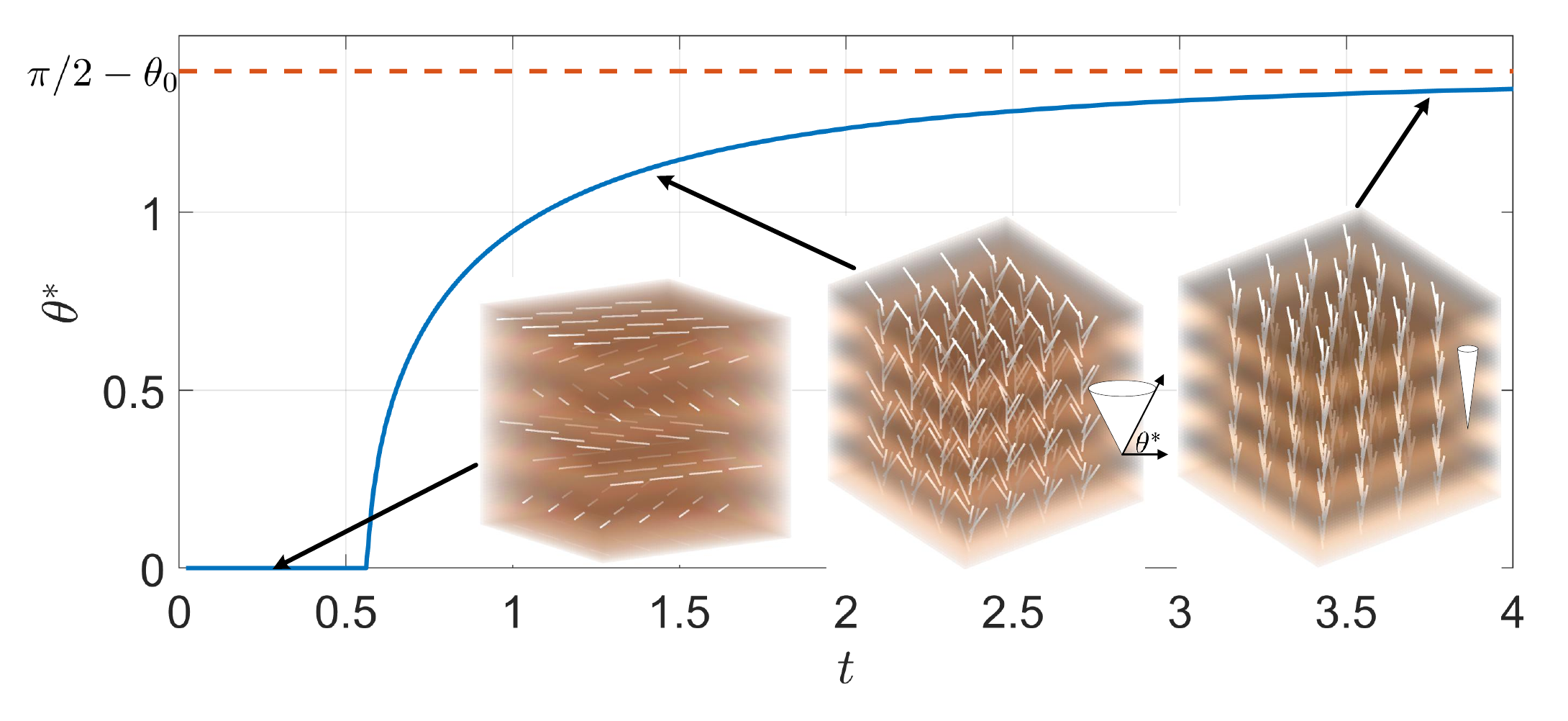}
\caption{$\theta^*$ versus $t$ in \eqref{optimal twist plane} which represents the Helical smectic-Smectic C* phase transition with increasing $t$ (decreasing temperature). The director (represented by the white line segments) twists within the conical plane determined by $\theta^*$
, while the layer structure is visualized by $tsin(qz)$. Therefore, the larger the value of $t$, the more pronounced the layer structure becomes. 
}
\label{fig: optimal twist plane}
\end{figure}

\section{Numerical results}\label{sec: numerical result}
\label{sec:num}
In this section, we perform numerical experiments to validate our theoretical results and understand the symmetry-breaking transitions from the cholesteric phase to the helical smectic phase and to the Smectic-C* phase. We study the phase transition problem with boundary conditions in \eqref{admissible space of theta and dr}, thus, we use the cosine and sine spectral method \cite{shen2011spectral} for spatial discretization,
\begin{equation}
    \theta(z)=
        \sum_{k=0}^{N+1} \Tilde{u}_k \cos\left(2k\pi x/h \right),
        \dr(z)=
        \sum_{k=1}^{N+1} \Tilde{u}_k \sin\left(2k\pi x/h \right), 
    \label{spectral method for u}
\end{equation}
where $N$ is an even integer and we choose $N=64$. Recall that $W_\dr \subset W^{1,2}_{0,\Omega}$, so we use the sine spectral method to discretize $\dr \in W_\dr$. By substituting \eqref{spectral method for u} into \eqref{energy theta final}, we obtain a discretized form of the energy,
\begin{equation*}
    F(\theta_k,\dr_k)\approx F(\theta,\dr).
\end{equation*}
This results in a function of $2N+3$ variables and we directly search for the minimum by using the gradient descent method for finite-dimensional functions, $F(\theta_k,\dr_k)$.

\begin{figure}
\centering
\includegraphics[width=.9\textwidth]{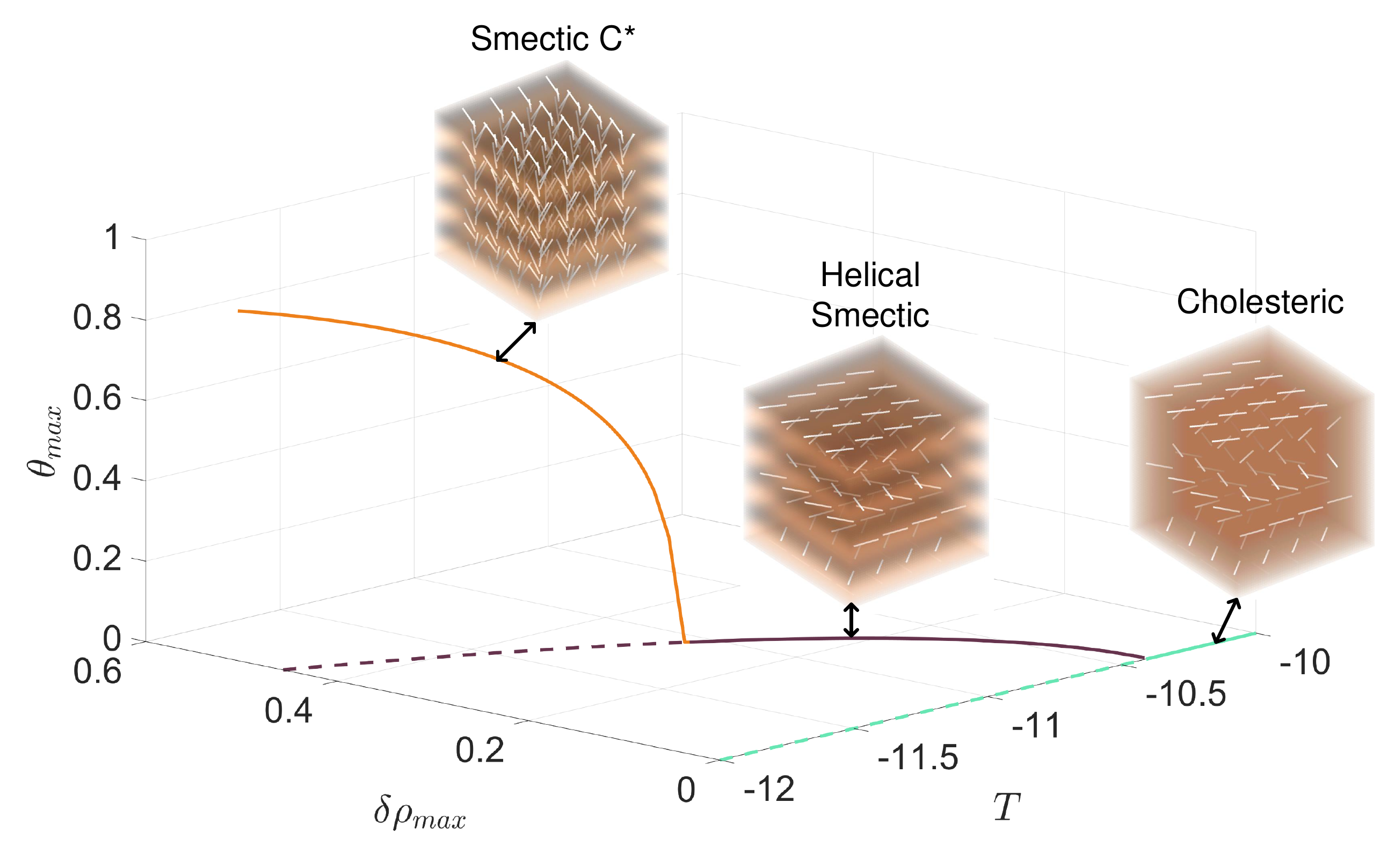}
\caption{Cholesteric--Helical Smectic--Smectic C* phase transitions of the energy functional \eqref{energy theta final} for $T_2^*=-10$, $\alpha_2=1$ (i.e. $d=T+10$), $f=10$, $h=2\pi$, $q=\sigma=4$, $\lambda_1=\lambda_2=0.001$, $\theta_0=\pi/9$. We use $\dr_{max}(T)$ and $\theta_{max}(T)$, where $\dr_{max}(T)$ = $\max_{0\leqslant x \leqslant h} \dr^*_T(x)$ and $\theta_{max}(T)$ = $\max_{0\leqslant x \leqslant h} \theta^*_T(x)$, to represent the global minimizer $(\dr^*_T(x),\theta^*_T(x))$ of $F(\theta,\dr)$ in \eqref{energy theta final} at $T$. For better visualisation, we plot the 3D $xy$ invariants: $\Tilde{\theta}(x,y,z)\equiv \theta(z)$ and $\Tilde{\dr}(x,y,z) \equiv \dr(z)$.}
\label{Phase transition as temperature}
\end{figure}

The numerically calculated bifurcation diagram of the Cholesteric-Helical Smectic-Smectic C* phase transition versus temperature $T$ is shown in Figure \ref{Phase transition as temperature}. The cholesteric phase with $\dr \equiv 0$ and $\theta \equiv 0$ is always a critical point of \eqref{energy theta final}. When $T\geqslant T_{C-HS}^*\approx -10.4$, the cholesteric phase is a global minimizer of \eqref{energy theta final} (Proposition \ref{cholesteric loses stability}). For $T_{C-HS}^* > T \geqslant T_{HS-SC}^*\approx -11.06$, the cholesteric phase loses stability (Proposition \ref{cholesteric loses stability}), and bifurcates into the stable helical smectic phase with $\dr \not\equiv 0$ and $\theta \equiv 0$ (Proposition \ref{bifurcation proof}). For $T< T_{HS-SC}^*$, the helical smectic phase loses stability, and the smectic C* phase with $\dr \not\equiv 0$ and $\theta \not\equiv 0$ becomes stable (Proposition \ref{twist plane changes}). The effect of the elastic constants on the optimal twist plane is shown in Figure \ref{effect of elastic energy}, which shows the competition between the nematic elastic constants, $k_i$, and the smectic elastic constants, $\lambda_i$. If the nematic elastic constants are sufficiently large or the smectic elastic constants are sufficiently small, the stable configuration should exhibit a very small (approximately zero) out-of-plane angle (Proposition \ref{helical configuration limit}). In the left plot of Figure \ref{effect of elastic energy}, we fix the smectic elastic constants and increase the nematic elastic constants; the out-of-plane twist angle approaches zero as the nematic elastic constants increase in magnitude. In the right plot of Figure \ref{effect of elastic energy}, we fix the nematic elastic constants and increase the smectic elastic constants; the out-of-plane twist angle is an increasing function of $\lambda_i$ consistent with the fact that the smectic energy favours out-of-plane twistedness.




\begin{figure}
\centering
\includegraphics[width=.95\textwidth]{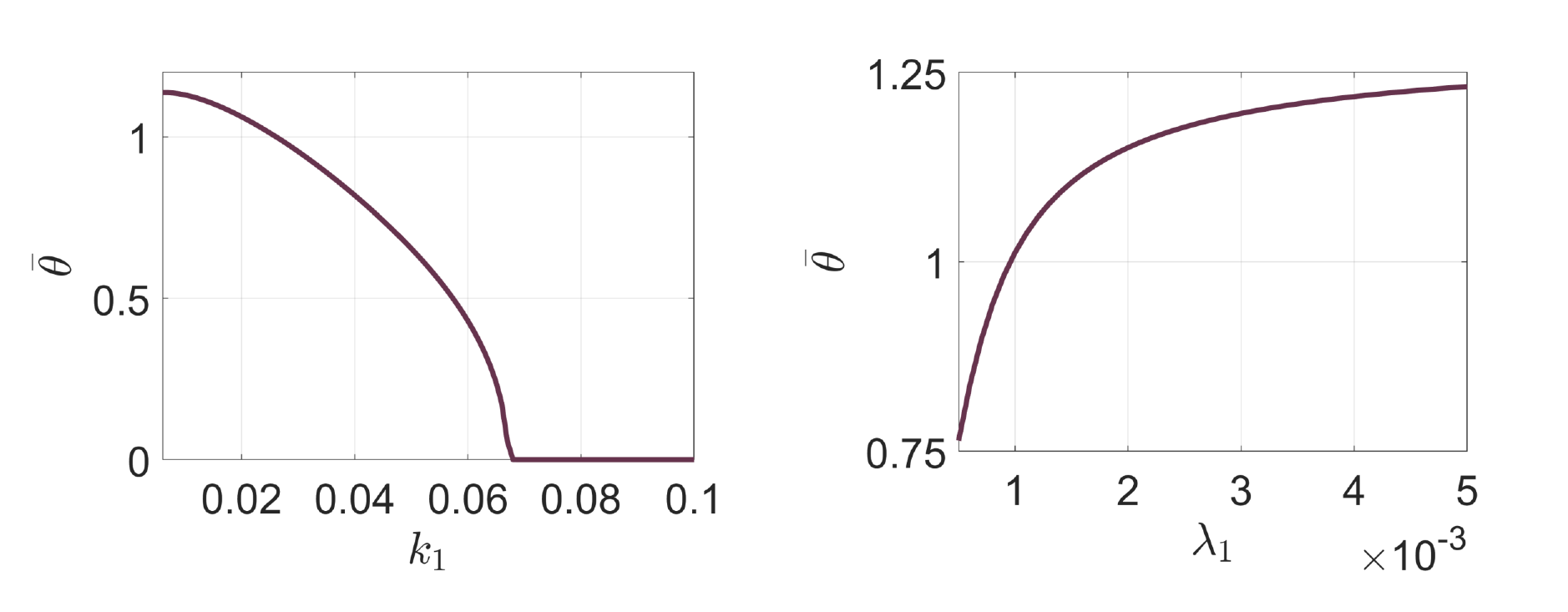}
\caption{The effect of the nematic elastic constant and smectic elastic constant on the stable twist plane. $\bar{\theta}$ is the average out-of-plane angle, which is defined as $\bar{\theta}=\int_0^h \theta^*(z)/h \ \d z$, where $(\dr^*,\theta^*)$ is the global minimizer of $F(\theta,\dr)$ in \eqref{energy theta final}. In the left plot, we take parameters $k_1=k_2$, $d=-5, f=10, \lambda_1=\lambda_2=0.001, h=2\pi, q=\sigma=4$, and $\theta_0=\pi/9$. In the right plot, we take $\lambda_1=\lambda_2$, $d=-5, f=10, k_1=k_2=k_3=0.025, h=2\pi, q=\sigma=4, \theta_0=\pi/9$.} 
\label{effect of elastic energy}
\end{figure}

\section{Conclusions}
\label{sec:conclusions}
We analyze a Landau-de Gennes type model for studying the symmetry-breaking transitions between cholesteric, helical smectic, and smectic C* phases. This model has two order parameters: the canonical Landau-de Gennes (LdG) order parameter and the positional order parameter, $\dr$, to account for the smectic layering. The corresponding free energy has multiple contributions - the standard LdG free energy of cholesterics, including a bulk potential and the cholesteric elastic energy density, a smectic bulk energy density, the smectic layering energy and a coupling energy between $\Qvec$ and $\dr$. The coupling energy dictates the preferred angle between the nematic director (modelled by the leading eigenvector of $\Qvec$ with largest positive eigenvalue) and the normal to the smectic layers. For the smectic A phase, the director and the smectic layer normal are co-aligned whereas a smectic C* phase is characterized by a non-zero angle between the director and the normal to the smectic layers. We prove the existence of global minimizers of this modified LdG-type energy in three-dimensional scenarios, with Dirichlet boundary conditions for $\Qvec$ and $\dr$.

We study minimizers of this modified LdG-type energy in three dimensions, in two asymptotic limits. The first asymptotic limit is the Oseen-Frank limit for which the re-scaled LdG bulk constants are larger than the re-scaled smectic constants, $d, e, f, \lambda_1, \lambda_2$ and the nematic elastic constants. In this limit, the modified LdG energy minimizers converge strongly to minimizers of the LdG bulk energy for $A<0$, defined by
\[
\Qvec = s_+ \left(\nvec \otimes \nvec - \frac{\mathbf{I}}{3}\right),
\] where $\nvec \in \mathbb{S}^2$, and $s_+$ is explicitly defined in terms of the LdG bulk constants. The Oseen-Frank limit is also equivalent to the large domain limit, for which the geometric length scales are much larger than the nematic coherence/correlation length.
The second asymptotic limit is, informally speaking, the Oseen-Frank limit combined with the small smectic energy limit, for which the rescaled parameters in \eqref{non-dimensionalized} satisfy
\[|A|, B, C \gg k_1, k_2, k_3 \gg |d|, e,f, \lambda_1, \lambda_2 \]
for which
$\nvec = \nvec_\sigma$ or the helical profile, to leading order. In other words, we recover the helical smectic phase in this second asymptotic limit for which the energy minimizer has (in an appropriately defined sense) a helical director profile and the positional order parameter or the smectic order parameter, $\dr$, may be zero or non-zero. A smectic phase corresponds to $\dr \not\equiv 0$ and a classical cholesteric phase corresponds to $\dr \equiv 0$.
Then, we study phase transitions as a function of the smectic bulk constant, $d$ (with a linear dependence on the temperature), in a one-dimensional setting, in the Oseen-Frank limit. We do not make a priori assumptions about the nematic elastic constants in this analysis. The limiting Oseen-Frank energy has two dependent variables: the angle $\theta$ between $\nvec$ and the $z$-axis and the smectic order parameter, $\dr$, and we study (local or global) minimizers of the limiting Oseen-Frank energy as a function of $d$. The cholesteric phase is defined by $\theta \equiv 0$ (in-plane director) and $\dr \equiv 0$. Note that when $\theta \equiv 0$, we recover $\nvec = \nvec_\sigma$ or the helical director profile regardless of the choice of $k_1, k_2$ (provided that they are consistent with \eqref{assumption oseen frank}). The helical smectic phase is defined by $\theta\equiv 0$ and $\dr \not\equiv 0$, and the Smectic C* phase is defined by $(\theta \not\equiv 0, \dr \not\equiv 0)$. We show that the cholesteric phase loses stability as $d$ decreases (temperature decreases), and undergoes a pitchfork bifurcation to the helical smectic phase. We show that the helical smectic phase loses stability as $d$ decreases, and in particular, $\theta \approx \frac{\pi}{2}-\theta_0$ when we are deep in the Smectic-C* phase (assuming an ansatz for $\dr$ and a constant $\theta$). In the limit of large elastic constants, we recover the helical director profile. 

The results in Section~\ref{sec: phase transition} are based on numerous assumptions but we numerically compute the critical points of \eqref{energy theta final} in Section~\ref{sec: numerical result}. These numerical results corroborate the Cholesteric-Helical Smectic-Smectic C* phase transitions as a function of decreasing temperature; they illustrate that the smectic layers become more pronounced as the temperature decreases and importantly, the numerical results demonstrate that $\theta \to \frac{\pi}{2} - \theta_0$ as $k_i$ get small or the smectic elastic constants get large (see Figure~\ref{effect of elastic energy} with $\theta_0=\frac{\pi}{9}$ so that $\theta \to 1.22$ in the limit of large smectic constants or in the low temperatures limit in Figure~\ref{Phase transition as temperature}). In some sense, our heuristic analysis in Section~\ref{sec: phase transition} is a good approximation to the full numerical results in Section~\ref{sec: numerical result}. These findings highlight the critical role of temperature in modulating the equilibrium states of chiral liquid crystals, offering insights into their thermodynamic behavior via the tensorial model.

Our work provides a robust theoretical and computational framework for studying complex phase transitions in liquid crystals, including the effects of chirality and positional ordering. Future work includes studying solution landscapes of such modified LdG energies, i.e., detailed studies of stable and unstable critical points of \eqref{eq:energy_LdG}. For example, can we have an unstable cholesteric phase with $\dr=0$ for low temperatures, co-existing with the Smectic C* phase, or can we have different types of stable smectic critical points, defined by different $\theta$ profiles and $\dr \neq 0$, in some suitably defined parameter regime? What are the experimental signatures of complex chiral smectic phases and what is the role of defects in stabilising frustrated chiral smectic systems? The modified LdG energy has many tunable phenomenological parameters, and the solution landscape is complex and potentially frustrated. It remains a challenge to fit these parameters to experimental data, and such fitting is essential for predictive modelling. We speculate that machine learning can play a constructive role in parameter fitting and texture identification, and future work might involve integrating machine learning methods with variational techniques for new-age liquid crystal research.



\section*{Acknowledgments}
The work of LZ was supported by the National Natural Science Foundation of China (No. 12225102, T2321001, 12288101). The work of JX was supported by the Young Elite Scientists Sponsorship Program by CAST (Grant No.~2023-JCJQ-QT-049), the Science and Technology Innovation Program of Hunan Province (Grant No.~2023RC3013), and the National Natural Science Foundation of China (Grant No.~12201636). The work of AM was supported by the Leverhulme Research Project Grant RPG-2021-401. BS would like to thank the Beijing International Center for Mathematical Research of Peking University and the University of Strathclyde for their support when the work on this paper was undertaken.

\bibliography{main}

\end{document}